\documentclass[12pt]{article}
\usepackage{amssymb, enumerate}
\usepackage{latexsym}
\usepackage{url}
\usepackage{amsthm}

\usepackage[all,cmtip]{xy}
\usepackage{amsmath,amssymb}
\usepackage{theoremref}

\newtheorem{theorem}{Theorem}[section]
\newtheorem{lemma}[theorem]{Lemma}
\newtheorem{corollary}[theorem]{Corollary}
\newtheorem{proposition}[theorem]{Proposition}

\theoremstyle{definition}

\newtheorem{example}[theorem]{Example}
\newtheorem{remark}[theorem]{Remark}

\newcommand{\Irr}{\mathop{\rm Irr}\nolimits}

\newcommand{\Hom}{\mathop{\rm Hom}\nolimits}

\newcommand{\Aut}{\mathop{\rm Aut}\nolimits}
\newcommand{\Out}{\mathop{\rm Out}\nolimits}

\newcommand{\Pic}{\mathop{\rm Pic}\nolimits}

\newcommand{\Perf}{{\rm Perf}}

\newcommand{\Id}{{\rm Id}}
\newcommand{\CF}{{\rm CF}}
\newcommand{\prj}{{\rm prj}}

\newcommand{\NN} {\mathbb{N}}
\newcommand{\ZZ} {\mathbb{Z}}
\newcommand{\cA} {\mathcal{A}}
\newcommand{\cO} {\mathcal{O}}
\newcommand{\cT} {\mathcal{T}}
\newcommand{\cL} {\mathcal{L}}
\newcommand{\cF} {\mathcal{F}}
\newcommand{\cB} {\mathcal{B}}

\def\L{\Lambda}

\def\bigcp{\mathop{\mathchoice 
 {\hbox{\sf\Large\lower 0.1\baselineskip\hbox{Y}}}%
 {\hbox{\sf\large\lower 0.1\baselineskip\hbox{Y}}}%
 {\hbox{\sf\normalsize\lower 0.1\baselineskip\hbox{Y}}}%
 {\hbox{\sf\tiny\lower 0.1\baselineskip\hbox{Y}}}%
}}
\def\bigtimes{\mathop{\mathchoice 
 {\hbox{\sf\Large\lower 0.1\baselineskip\hbox{X}}}%
 {\hbox{\sf\large\lower 0.1\baselineskip\hbox{X}}}%
 {\hbox{\sf\normalsize\lower 0.1\baselineskip\hbox{X}}}%
 {\hbox{\sf\tiny\lower 0.1\baselineskip\hbox{X}}}%
}}

\def\Sym(#1){\mathop{\rm Sym}(#1)}
\def\Sym(#1){S_{#1}}
\def\diag(#1){\mathop{\rm diag}(#1)}

\newenvironment{enumerate*}{%
 \begin{enumerate}%
 }%
 {\end{enumerate}}

\begin{document}

\title{Morita equivalence classes of $2$-blocks with abelian defect groups of rank $4$  \footnote{This research was supported by the EPSRC grants no. EP/M015548/1 and EP/T004606/1.}}

\author{Charles W. Eaton\footnote{Department of Mathematics, University of Manchester, Manchester M13 9PL. Email: charles.eaton@manchester.ac.uk} and Michael Livesey\footnote{Department of Mathematics, University of Manchester, Manchester M13 9PL. Email: dr.m.livesey@gmail.com}}

\maketitle


\begin{abstract}
We classify all $2$-blocks with abelian defect groups of rank $4$ up to Morita equivalence. The classification holds for blocks over a suitable discrete valuation ring as well as for those over an algebraically closed field. An application is that Brou\'{e}'s abelian defect group conjecture holds for all blocks under consideration here.

\end{abstract}


\section{Introduction}

Let $p$ be a prime and $(K,\cO,k)$ be a $p$-modular system with $k$ algebraically closed, and let $P$ be a finite $p$-group. It is known by work culminating in~\cite{eel20} that if $P$ is an abelian $2$-group, then there are only finitely many Morita equivalence classes amongst all blocks of $\cO G$ for all finite groups $G$ with defect group $D$ isomorphic to $P$, i.e., that Donovan's conjecture holds for such $P$. This suggests the problem of classifying the Morita equivalence classes of blocks that arise for given abelian $2$-groups $P$. The Morita equivalence classes are already known without the use of the classification when $P$ is abelian and $\Aut(P)$ is a $2$-group and when $P$ is a Klein four group. In the former case every block must be nilpotent by~\cite{pu88} since $P$ controls fusion for the block, and in the latter there are three Morita equivalence classes by~\cite{li94} (in fact these are the only source algebra equivalence classes by~\cite{cekl11}). In~\cite{ekks14} the classification of finite simple groups was applied to describe the $2$-blocks of quasisimple groups with abelian defect groups. This has been used to classify the Morita equivalence classes of blocks with defect groups isomorphic to $P$ in the following cases. When $P$ is elementary abelian of order $16$ or $32$, there are $16$ or $34$ Morita equivalence classes by respectively~\cite{ea19} and~\cite{ar21}. When $P$ is elementary abelian of order $64$, there are $81$ Morita equivalence classes containing a principal block by~\cite{ar20}. Morita equivalence classes of blocks when $P$ is an abelian $2$-group of rank at most three are determined in~\cite{el18} and~\cite{wzz18}. Further classifications are obtained when we place restrictions on the inertial quotient of the blocks we consider, as in~\cite{mck20} and~\cite{am20}, where it is assumed that the inertial quotient contains an element of maximal order, or is a subgroup of a cyclic group of maximal order. Note that for odd primes, full classifications of Morita equivalence classes of blocks with abelian defect groups are currently only known when $P$ is a cyclic $3$-group or $C_5$. Morita equivalence classes have also been determined for some classes of nonabelian defect groups, which for completeness we briefly list here. Of the tame blocks with nonabelian defect groups, complete classifications are only known for dihedral $2$-groups (over $k$, for which see~\cite{er87} and~\cite{ma22}) and $Q_8$ (over $\cO$, for which see~\cite{ei16}). Aside from $p$-groups with only one saturated fusion system (for which blocks must be nilpotent), Morita equivalence classes of blocks with respect to $\cO$ have been determined when $P$ is a Suzuki $2$-group (see~\cite{ea24}), an extraspecial $p$-group of order $p^3$ and exponent $p$ for $p \geq 5$ (see~\cite{ae23}), or is in a class of minimal nonabelian $2$-groups as in~\cite{eks12}. 

Here we determine the Morita equivalence classes of blocks when $p=2$ and $P$ is abelian of rank at most four. There are two main problems to overcome in the reduction to quasisimple groups in order to apply~\cite{ekks14}: the cases of a normal subgroup of odd index and of a normal subgroup of index two. In the former case we use Picard groups and crossed products to deduce possible Morita equivalence classes for the block of the over group. In the latter we use a combination of a method developed in~\cite{wzz18} and one of our own.  

Let $D \cong C_{2^{n_1}} \times C_{2^{n_2}} \times C_{2^{n_3}} \times C_{2^{n_4}}$, where $n_1,\ldots,n_4 \geq 0$. The conjugacy classes of odd order subgroups $E$ of $\Aut(D)$ have representatives as given in Table \ref{Inertial_quotient:table}, and correspond to the possible inertial quotients for blocks with defect group $D$ (where by inertial quotient, defined later, we consider also the action on $D$). By~\cite[5.2.3]{gor} we may write $D=[D,E] \times C_D(E)$. To simplify notation we assume that labeling is chosen so that $[D,E] \cong C_{2^{n_1}} \times \cdots \times C_{2^{n_i}}$ and $C_D(E) \cong C_{2^{n_{i+1}}} \times \cdots \times C_{2^{n_4}}$ for some $i$. If $B$ is a block with defect group $D$ and inertial quotient $E$, then we say that $B$ is of type $E$, where we use the notation of Table \ref{Inertial_quotient:table} to distinguish between isomorphic but non-conjugate subgroups of $\Aut(D)$. Note that the case $(C_3)_2$ represents the simultaneous action of $C_3$ on $C_{2^{n_1}} \times C_{2^{n_2}}$ and on $C_{2^{n_3}} \times C_{2^{n_4}}$, which for  $n_1=n_2=n_3=n_4$, also represents a subgroup of $C_{15}$. 

\begin{table}
\centering
\begin{tabular}{|l|l|l|}
\hline
Inertial quotient & Notes & Restrictions on $D$ \\
\hline
$1$ & & None \\
$(C_3)_1$ & $C_D(E)$ has rank $2$ & $n_1=n_2$ \\
$(C_3)_2$ & $C_D(E)=1$ & $n_1=n_2$, $n_3=n_4$ \\
$C_3 \times C_3$ & $C_D(E)=1$ & $n_1=n_2$, $n_3=n_4$ \\
$C_5$ & $C_D(E)=1$ & $n_1=n_2=n_3=n_4$ \\
$C_{15}$ & $C_D(E)=1$ & $n_1=n_2=n_3=n_4$ \\
$C_7$ & $C_D(E)$ cyclic & $n_1=n_2=n_3$ \\
$C_7 \rtimes C_3$ & $C_D(E)$ cyclic & $n_1=n_2=n_3$ \\
\hline
\end{tabular}
\caption{Possible inertial quotients}
\label{Inertial_quotient:table}
\end{table}

We now give the classification. Throughout this paper, $G_n$ will denote the nonabelian group $(C_{2^n} \times C_{2^n}) \rtimes C_3$. 

Following~\cite{pu11} a block is called \emph{inertial} if it is basic Morita equivalent to a block with normal defect group.

\begin{theorem}
\thlabel{main_theorem}
Let $G$ be a finite group and $B$ be a block of $\cO G$ with defect group $D \cong C_{2^{n_1}} \times C_{2^{n_2}} \times C_{2^{n_3}} \times C_{2^{n_4}}$ as above. Let $E$ be an inertial quotient of $B$, with notation as in Table \ref{Inertial_quotient:table}.

\begin{enumerate}[(a)]
\item\thlabel{singer} If $n_1=n_2$, $n_3=n_4$ and $E$ acts as $1$, $(C_3)_2$ or $C_5$, then $B$ is inertial.
\item \begin{enumerate}[(i)]
 \item If $E$ acts as $(C_3)_1$, then $B$ is Morita equivalent to one of $\cO(D \rtimes E)$ or $B_0(\cO(A_5 \times C_{2^{n_3}} \times C_{2^{n_4}}))$.
 \item If $E$ acts as $C_7$, then $B$ is Morita equivalent to $\cO (D \rtimes C_7)$ or $B_0(\cO(SL_2(8) \times C_{2^{n_4}}))$.
 \item If $E$ acts as $C_7 \rtimes C_3$, then $B$ is Morita equivalent to $\cO (D \rtimes (C_7 \rtimes C_3))$, $B_0(\cO(\Aut(SL_2(8)) \times C_{2^{n_4}}))$ or $B_0(\cO (J_1 \times C_{2^{n_4}}))$.
  \item If $E$ acts as $C_{15}$, then either $B$ is inertial or $B$ is basic Morita equivalent to $B_0(\cO SL_2(16))$.
  \item Suppose $E$ acts as $C_3 \times C_3$, with $n_1=n_2$ and $n_3=n_4$. Then $B$ is Morita equivalent to one of $B_0(\cO(G_{n_1} \times G_{n_3}))$, $B_0(\cO(A_5 \times G_{n_3}))$, $B_0(\cO(A_5 \times A_5))$ or a nonprincipal block of $D \rtimes 3_+^{1+2}$, where $Z(3_+^{1+2})$ acts trivially on $D$.
\end{enumerate}
\end{enumerate}
\end{theorem}

\begin{remark}
(i) Nonprincipal blocks of $D \rtimes 3_+^{1+2}$ and $D \rtimes 3_-^{1+2}$ as above will be shown to be Morita equivalent in \thref{faithfulblocks}.

(ii) Implicit in the statement of \thref{main_theorem} is the assertion that for these defect groups Morita equivalent blocks have the same fusion (i.e., the same inertial quotient).

(iii) Cases (a) and (b)(iv) of \thref{main_theorem} have been proved in~\cite{mck20} and~\cite{am20}.
\end{remark}


A consequence is the following, that Brou\'{e}'s abelian defect group conjecture holds for the blocks under consideration:

\begin{corollary}
\thlabel{broue}
Let $G$ be a finite group and $B$ be a $2$-block of $\cO G$ with abelian defect group $D$ of rank at most four. Let $b$ be the Brauer correspondent of $B$ in $\cO N_G(D)$. Then $B$ is derived equivalent to $b$.
\end{corollary}

Since Brou\'e's conjecture is known for blocks with elementary abelian defect groups of order $32$ by~\cite{as21}, we immediately have that the conjecture holds for all blocks with abelian defect groups of order dividing $32$.

\begin{remark}
Brou\'{e}'s abelian defect group conjecture is often stated as requiring further a splendid Rickard equivalence. Since \thref{main_theorem} only provides a Morita equivalence with no additional information on the bimodule affording the equivalence, we make no claim to the existence of a splendid equivalence in our result.
\end{remark}

The structure of the paper is as follows. In Section \ref{Background} we briefly cover necessary notation and general results. This includes: results on covering of blocks of normal subgroups and related Morita equivalences; subpairs and inertial quotients; a review of the different (stronger) versions of Morita equivalence that we use, and their relationship with fusion; a summary of the results of~\cite{ekks14} that we use; and Picard groups. We prove a result on extending Morita equivalences (in certain cases) from normal subgroups of index $2$ in Section \ref{index2compare}. This requires knowledge of some groups of perfect self-isometries, which is the content of Section \ref{PI}. In Section \ref{extensions} we give results on extending Morita equivalences from normal subgroups, which includes extensions using crossed products where the index is odd, combined with techniques for index $2$ from the previous section and the result developed in~\cite{wzz18}. In Section \ref{main_proof} we present the proof of the main result. Finally in Section \ref{broue_section} we prove \thref{broue}.


\section{Notation and background}
\label{Background}

\subsection{Morita equivalence}

See~\cite[2.8]{lin1} for an introduction to Morita theory. It is not known whether Morita equivalence of blocks of finite groups defined over $\cO$ preserve the isomorphism type of the defect group (Morita equivalences defined over $k$ are known not to necessarily preserve this by~\cite{gma22}). However, by~\cite{be89} Morita equivalence does preserve the isomorphism class of abelian defect groups. It is also not known whether the inertial quotient (and its action) are preserved. We will have to take some care in the proof of \thref{main_theorem} to check that this does happen in our situation.

A Morita equivalence is \emph{basic} if it is induced by an endopermutation source bimodule, and a block is \emph{inertial} if there is basic Morita equivalence with the Brauer correspondent block of the normalizer of a defect group (see~\cite{pu99} and~\cite{pu11}). A basic Morita equivalence preserves the isomorphism type of the defect group and fusion.  

We will also refer to source algebra equivalences (see~\cite[6.4]{lin2}). These occur in the section on Picard groups. Also certain Morita equivalences that we use are stated as source algebra equivalences (for example \thref{solvablequotient:lem}), but we do not use properties beyond that they are basic Morita equivalences.

Another point to mention here is the Bonnaf\'{e}-Rouquier correspondence as in~\cite{bdr17}. Whilst correspondent blocks are only known to be Morita equivalences (i.e, not basic), they are splendid Rickard equivalent, which implies they have the same fusion.


\subsection{Blocks and normal subgroups}

We collect some background on blocks and normal subgroups that will be used frequently and usually without further reference. A reference for this material is~\cite[Section 6.8]{lin2}

Let $G$ be a finite group and $B$ be a block of $\cO G$ with defect group $D$. We use $\prj(B)$ to denote the set of characters of projective indecomposable $B$-modules. Let $N \lhd G$ and let $b$ be a block of $\cO N$ covered by $B$. We may choose a $G$-conjugate of $D$ such that $D \cap N$ is a defect group for $b$. Write ${\rm Stab}_G(b)$ for the stabiliser of $b$ under conjugation in $G$. There is a block of $\cO G$ covering $b$ with a defect group $P$ such that $NP/N$ is a Sylow $p$-subgroup of ${\rm Stab}_G(b)/N$. 

If $[G:N]$ is a power of $p$, then $B$ is the unique block of $\cO G$ covering $b$, and it follows that $b$ is $G$-stable if and only if $G=ND$. In this case $B$ and $b$ share a block idempotent. We have the following crucial fact in the case that $D$ is abelian:

\begin{proposition}
\thlabel{Dabelian_inner:prop}
Let $G$ be a finite group and $B$ a block of $\cO G$ with abelian defect group $D$. Suppose that $N \lhd G$ with $G=ND$ and that $b$ is a $G$-stable block of $\cO N$ covered by $B$. Then $D$ acts as inner automorphisms on $b$. Further,

(i) every irreducible character of $b$ is $G$-stable, hence when further $[G:N]=p$, extends to $p$ distinct irreducible characters of $B$,

(ii) induction gives a bijection between the projective indecomposable modules for $b$ and those for $B$.
\end{proposition}

\begin{proof}
The assertion that $D$ acts as inner automorphisms follows from~\cite[Corollary 6.16.3]{lin2} or~\cite[Proposition 3.1]{el19b}. The rest is~\cite[Proposition 2.6(i), (ii)]{el18}.
\end{proof}

Note that the version of the first part of Proposition \ref{Dabelian_inner:prop} over $k$ is also found within~\cite{kk96}.

Recall that a block $B$ is \emph{quasiprimitive} if every block of every normal subgroup covered by $B$ is $G$-stable. In particular $B$ covers a unique block for each normal subgroup. 

We will make frequent use of the following, especially in the case where the quotient group is a subgroup of the outer automorphism group of a quasisimple group.

\begin{lemma}[Lemma 2.4 of~\cite{ar21}]
\thlabel{solvablequotient:lem}
Let $G$ be a finite group and let $N \lhd G$ with $G/N$ solvable. Let $b$ be a $G$-stable block of $\cO N$ and let $B$ be a quasiprimitive block of $\cO G$ covering $b$ with defect group $D$. Then $DN/N$ is a Sylow $p$-subgroup of $G/N$.
\end{lemma}

The normal subgroup $G[b]$ of $G$ is defined to be the group of elements of $G$ acting as inner automorphisms on $b \otimes_{\cO} k$ (see~\cite{ku90}).

\begin{proposition}
\thlabel{innerautos}
Let $G$ be a finite group and $B$ a block of $\cO G$ with defect group $D$. Let $N \lhd G$ with $D \leq N$ and suppose that $B$ covers a $G$-stable block $b$ of $\cO N$. Let $B'$ be a block of $\cO G[b]$ covered by $B$. Then

(i) $b$ is source algebra equivalent to $B'$, and in particular has isomorphic inertial quotient;

(ii) $B$ is the unique block of $\cO G$ covering $B'$.
\end{proposition}

\begin{proof}
Part (i) is~\cite[2.2]{kkl12}, noting that a source algebra equivalence over $k$ implies one over $\cO$ by~\cite[7.8]{pu88}. Part (ii) follows from~\cite[3.5]{da73}.
\end{proof}

\subsection{Subpairs, fusion and inertial quotients}

Let $G$ be a finite group and $B$ be a block of $\cO G$ with defect group $D$. For convenience in stating definitions we assume that $D$ is abelian. A $B$-subpair is a pair $(Q,B_Q)$ where $Q$ is a $p$-subgroup of $G$ and $B_Q$ is a block of $C_G(Q)$ with Brauer correspondent $B$. Such a block $B_Q$ only exists when $Q$ is $G$-conjugate to a subgroup of $D$. For a fixed $B$-subpair $(D,B_D)$, for each $Q \leq D$ there is an unique $B$-subpair $(Q,B_Q)$ such that $B_Q$ is the Brauer correspondent of $B_D$. 

The $B$-subpairs define a fusion system $\mathcal{F}=\mathcal{F}_B(D,B_D)$, often called the \emph{Frobenius category} of the $B$ (see for example~\cite{ako} or~\cite{lin1}). The \emph{inertial quotient} of $B$ is $E=N_G(D,B_D)/DC_G(D)$, together with the action of $E$ on $D$. Since $D$ is abelian, $\cF$ is determined by $E$. Consequently, every block with abelian defect groups and trivial inertial quotient is nilpotent (we will use this fact throughout this paper without further reference).

The following is well-known.

\begin{proposition}
\thlabel{defect_group_factor:prop}
Let $B$ be a block of $\cO G$ for a finite group $G$ with abelian defect group $D$. Let $(D,B_D)$ be a $B$-subpair. Then $D=[D,N_G(D,B_D)] \times C_D(N_G(D,B_D))$.

Suppose there is $N \lhd G$ such that $G=DN$, then $[D,N_G(D,B_D)] \leq N$.

\end{proposition}

\begin{proof}
The factorisation of $D$ follows from~\cite[Theorem 5.2.3]{gor}. For the second part, note that $[D,N_G(D,B_D)] = [D,DN_N(D,B_D)] \leq N$.
\end{proof}

Most of the following is Lemma 5.3.5 of~\cite{McK}.

\begin{proposition}
\thlabel{index_p_inertial:prop}
Let $B$ be a block of $\cO G$ for a finite group $G$ with abelian defect group $D$  and let $(D,B_D)$ be a $B$-subpair. Let $N \lhd G$. Suppose that $B$ covers a $G$-stable block $b$ of $\cO N$, and that $G=DN$. Write $Q=N \cap D$, a defect group for $b$. Let $b_Q$ be a block of $\cO C_N(Q)$ covered by $B_Q$. Then $b_Q$ is $C_G(Q)$-stable and $(Q,b_Q)$ is a $b$-subpair. The blocks $B$ and $b$ have isomorphic inertial quotients and $[N_G(D,B_D),D]=[N_N(Q,b_Q),Q]$.
\end{proposition}

\begin{proof}
That $(Q,b_Q)$ is a $b$-subpair forms part of the proof of the main theorem of~\cite{kk96}. We note that due to the correspondence between $\cO$-blocks and $k$-blocks, the result may be proved for $k$-blocks, which is the setting for~\cite{kk96}. Note also that this part of the proof of the main result of~\cite{kk96} does not require that $G$ is a split extension of $N$. 

We have $C_G(Q)=DC_N(Q)$ and $N_G(Q)=DN_N(Q)$. Hence $N_G(Q,b_Q)=DN_N(Q,b_Q)$. Also, since $b_Q$ is $C_G(Q)$-stable, it is the unique block of $C_N(Q)$ covered by $B_Q$ (and $B_Q$ is the unique block of $C_G(Q)$ covering $b_Q$ since $C_N(Q)$ has index a power of $p$ in $C_G(Q)$). Hence $N_G(Q,B_Q)=N_G(Q,b_Q)$. By definition we have $N_G(D,B_D) \leq N_G(Q,B_Q)$.

Since $N_G(D,B_D)$ controls fusion in $D$, by~\cite[Proposition 4.24]{ab79} we have $N_G(Q,B_Q)=N_G(D,B_D)C_G(Q)$. It follows that $$N_N(Q,b_Q)/C_N(Q) \cong N_G(D,B_D)/(C_G(Q) \cap N_G(D,B_D)).$$ Noting that $D=QC_D(N_G(D,B_D))$, we have $C_G(D)=C_{N_G(D,B_D)}(D)=C_{N_G(D,B_D)}(Q)$, so $B$ and $b$ have isomorphic inertial quotients. 

Further $$[N_G(D,B_D),D]=[N_G(D,B_D),QC_D(N_G(D,B_D))]=[N_G(D,B_D),Q]$$ 
$$=[N_G(D,B_D)C_G(Q),Q]=[N_G(Q,B_Q),Q]=[N_G(Q,b_Q),Q]=[N_N(Q,b_Q),Q]$$ as required.
\end{proof}

The next result is essentially extracted from the proof of~\cite[Lemma 6.3]{wzz18}, and allows us to compare inertial quotients of blocks with those of blocks of normal subgroups of $p'$ index. We include the proof here for convenience.

\begin{lemma}
\thlabel{inertial_quotient_normal}
Let $B$ be a block of $\cO G$ for a finite group $G$, and let $b$ be a $G$-stable block of $\cO N$ for $N \lhd G$ with $[G:N] = r$ a prime different to $p$. Let $E_B$, $E_b$ be the inertial quotient of $B$, $b$ respectively. Let $D$ be an abelian defect group for $B$ (and for $b$). Then either $E_b$ is isomorphic to a subgroup of $E_B$ or $E_B$ is isomorphic to a subgroup of $E_b$. 
\end{lemma}

\begin{proof}
We may take a $B$-subpair $(D,B_D)$ and a $b$-subpair $(D,b_D)$ with $B_D$ covering $b_D$. If $C_G(D) \leq N$, then $B_D=b_D$ and $N_N(D,b_D) \leq N_G(D,B_D)$, so the result is immediate in this case. Hence suppose $C_G(D)$ is not contained in $N$, i.e,, $G=NC_G(D)$. 

If $C_G(D) \leq N_G(D,b_D)$, then $N_G(D,B_D) \leq N_G(D,b_D)=N_N(D,b_D)C_G(D)$. Hence $N_G(D,B_D)/C_G(D) \leq N_G(D,b_D)/C_G(D) \cong N_N(D,b_D)/C_N(D)$.

Suppose that $C_G(D)$ does not stabilize $b_D$. Since $[C_G(D):C_N(D)]=r$ is a prime and $C_N(D) \leq C_G(D) \cap N_G(D,b_D) < C_G(D)$, we have $C_N(D) = C_G(D) \cap N_G(D,b_D)$. Let $T$ be a transversal of $C_N(D)$ in $C_G(D)$. Then $\{ b_D^t:t \in T\}$ is the set $r$ distinct conjugate blocks of $C_N(D)$ covered by $B_D$ and $B_D$ is the unique block of $C_G(D)$ covering $b_D$. Hence $N_G(D,B_D) = N_G(D,b_D)C_G(D)$. Then $N_G(D,B_D)/C_G(D) \cong N_G(D,b_D)/C_N(D)$ and the result follows.
\end{proof}


\subsection{Blocks of quasisimple groups}

We extract the results of~\cite{ekks14} necessary for this paper. Recall that a block $B$ of a group $G$ is nilpotent covered if there is $H$ with $G \lhd H$ and a nilpotent block of $H$ covering $B$. Properties of such blocks are considered in~\cite{pu11}. In particular nilpotent covered blocks are inertial by~\cite[Corollary 4.3]{pu11}.

\begin{proposition}[\cite{ekks14}]
\thlabel{qsclassification}
Let $p=2$, and let $B$ be a block of $\cO G$ for a quasisimple group $G$ with abelian defect group $D$ of rank at most $4$. Then one or more of the following occurs:

(i) $G \cong A_5$, $SL_2(8)$, $SL_2(16)$, $J_1$ or ${}^2G_2(q)$, where $q=3^{2m+1}$ for some $m \in \NN$, and $B$ is the principal block;

(ii) $G \cong Co_3$ and $B$ is the unique non-principal $2$-block of defect $3$;

(iii)  $B$ has inertial quotient of type $(C_3)_1$ and is Morita equivalent to a block $C$ of $\cO L$ where $L =L_0 \times L_1 \leq G$ such that $L_0$ is abelian and the block of $\cO L_1$ covered by $C$ has Klein four defect groups;

(iv) $B$ is nilpotent covered.
\end{proposition}

\begin{proof}
This follows from Theorem 6.1 of~\cite{ekks14} and its proof. The only point to address is the inertial quotient in case (iii). Case (iii) arises in two ways. The first is where $D \cong C_2 \times C_2$, in which case either $B$ is nilpotent (and so is covered by case (iv)), or has inertial quotient $C_3$. The second way this arises is as in case (v) of~\cite[Proposition 5.3]{ekks14}. Here the given Morita equivalence is given by the Bonnaf\'{e}-Rouquier correspondence as in~\cite{bdr17}. Now the corresponding blocks in~\cite{bdr17} are equivalent by a splendid Rickard equivalence and so have the same fusion. Hence $B$ has the same inertial quotient as $C$, which is of type $(C_3)_1$.
\end{proof}

\begin{corollary}
\thlabel{Gn_inertial}
Every $2$-block of a quasisimple group that is Morita equivalent to $G_n$, for $n \in \NN$, is inertial. 
\end{corollary}

\begin{proof}
For $n \geq 2$ we see from \thref{qsclassification} that blocks of quasisimple groups with defect groups $(C_{2^n})^2$ are nilpotent covered, and so inertial. For $n=1$, by~\cite{cekl11} every block that is Morita equivalent to $\cO A_4$ is source algebra equivalent to $\cO A_4$, and so inertial.
\end{proof}


\subsection{Picard groups}
\label{Picard}

Let $G$ be a finite group and $B$ be a block of $\cO G$ with defect group $D$. The Picard group $\Pic(B)$ of $B$ has elements the isomorphism classes of $B$-$B$-bimodules inducing $\cO$-linear Morita auto-equivalences of $B$. For $B$-$B$-bimodules $M$ and $N$, the group multiplication is given by $M \otimes_B N$. For background and definitions we follow~\cite{bkl20}. 

We will use knowledge of $\Pic(B)$ to refine K\"ulshammer's analysis in~\cite{ku95} of the situation of a normal subgroup containing the defect groups of a block. This involves the study of crossed products of a basic algebra with a $p'$-group, which in turn uses the outer automorphism group of the basic algebra, a group that embeds into the Picard group. This will be essential in reduction steps in our classification of Morita equivalence classes of blocks.

Write $\cT(B)$ for the subgroup of $\Pic(B)$ consisting of bimodules with trivial source and $\cL(B)$ for the subgroup consisting of linear source modules.

$\cT(B)$ and $\cL(B)$ are described in~\cite{bkl20}, and we summarise the relevant notation and results here. Let $\cF = \cF_B(D,B_D)$ be the fusion system for $B$ on $D$, defined using a $B$-subpair $(D,B_D)$, and let $E=N_G(D,B_D)/DC_G(D)$ be the inertial quotient. Write $\Aut(D,\cF)$ for the subgroup of $\Aut(D)$ of automorphisms stabilizing $\cF$ and $\Out(D,\cF)=\Aut(D,\cF)/\Aut_\cF(D)$. By~\cite{el20}, if $D$ is abelian, then $\Out(D,\cF) \cong N_{\Aut(D)}(E)/E \cong \Out(D \rtimes E)$.

Let $A$ be a source algebra for $B$. Then $A$ is an interior $D$-algebra, so we have an embedding of $D$ into $A$ and we may consider the fixed points $A^D$ under the action of $D$. Write $\Aut_D(A)$ for the group of $\cO$-algebra automorphisms of $A$ fixing each element of the image of $D$ in $A$, and $\Out_D(A)$ for the quotient of $\Aut_D(A)$ by the subgroup of automorphisms given by conjugation by elements of $(A^D)^\times$. By~\cite[14.9]{pu88} $\Out_D(A)$ is isomorphic to a subgroup of $\Hom(E,k^\times)$. 

By~\cite[Theorem 1.1]{bkl20} we have exact sequences
\begin{equation}\label{arr:exact}\begin{array}{lllllll}
1 & \rightarrow & \Out_D(A) & \rightarrow & \cT(B) & \rightarrow & \Out(D,\cF), \\
1 & \rightarrow & \Out_D(A) & \rightarrow & \cL(B) & \rightarrow & \Hom(D/\mathfrak{foc}(D),\mathcal{O}^\times) \rtimes \Out(D,\cF), \\
\end{array}\end{equation} where $\mathfrak{foc}(D)$ is the focal subgroup of $D$ with respect to $\mathcal{F}$, generated by the elements $\varphi(x)x^{-1}$ for $x \in D$ and $\varphi \in \Hom_\mathcal{F}(\langle x \rangle,D)$. When $D$ is abelian we have $\mathfrak{foc}(D) = [N_G(D,B_D),D]$, so that $\Hom(D/\mathfrak{foc}(D),\mathcal{O}^\times) \cong C_D(N_G(D,B_D))$ (see \thref{defect_group_factor:prop}).

We record for later use that by~\cite{rs87}, if $P$ is a $p$-group, then $${\Pic}(\cO P) = \cL(\cO P) \cong{\Aut}(\cO P)\cong{\Hom}(P,\cO^\times)\rtimes{\Aut}(P).$$

We gather together here the results regarding Picard groups that we will use later on.

\begin{proposition}
\thlabel{pic_list}
Let $P$ and $Q$ be abelian $2$-groups and $n, n_1, n_2 \in \NN$. 

\begin{enumerate}[(i)]
\item\label{C_3Pic} $\Pic (\cO(G_{n} \times P)) \cong S_3 \times (P \rtimes \Aut(P))$. The subgroup of $\Pic (\cO(G_{n} \times P))$ given by those self-equivalences fixing the projective indecomposable modules is isomorphic to $P \rtimes \Aut(P)$.
\item\label{A5QPic} $\Pic(B_0(\cO (A_5 \times P))) \cong C_2 \times (P \rtimes \Aut(P))$.
\item\label{GnA5Pic} $\Pic(B_0(\cO (G_n\times A_5))) \cong S_3 \times C_2$.
\item\label{Gn1Gn2Pic} $\Pic(\cO(G_{n_1} \times G_{n_2}))  = \cT(\cO(G_{n_1} \times G_{n_2})) \cong$
\begin{enumerate}[(a)]
\item $S_3\wr S_2$ if $n_1=n_2$,
\item $S_3 \times S_3$ if $n_1\neq n_2$.
\end{enumerate}
\item\label{C7Pic} If $Q \cong (C_{2^n})^3$, then $\Pic(\cO ((Q \rtimes C_7) \times P)) \cong (C_7 \rtimes C_3) \times \Aut(P)$. 

\end{enumerate}

\end{proposition}

\begin{proof}
Throughout this proof, we denote by $D$ a defect group of the block in question.

(\ref{C_3Pic}) The Picard group is described in~\cite[Theorem 1.1]{el20}.  The $S_3$ factor consists of (bimodules inducing) Morita equivalences corresponding to elements of $C_3 \rtimes \Out(G_n)$, where the $C_3$ is generated by multiplying by a non-trivial linear character of $G_n$. The $P \rtimes \Aut(P)$ factor is generated by equivalences given by multiplication by a (linear) character and by automorphisms of $P$. Since the projective indecomposable modules correspond in this case to the irreducible characters with $D$ in their kernel, the remainder follows.

(\ref{A5QPic})-(\ref{Gn1Gn2Pic}) are from~\cite[Theorem 1.1]{el20}.

(\ref{C7Pic}) Suppose $G = (Q \rtimes C_7) \times P$ where $Q \cong (C_{2^n})^3$. Then $G=D \rtimes E$, where 
$E$ is the inertial quotient (note that $\cO G$ has a unique block). By~\cite{li21} $\Pic(\cO G)=\cL(\cO G)$.  
By~\cite[Lemma 2.1]{el20} $\Out(D,\cF)\cong N_{\Aut(D)}(E)/E\cong \Out(G) \cong C_3 \times (P \rtimes \Aut(P))$. Also it follows from~\cite[Lemma 2.2]{el20} that $\Out_D(A) \cong E$, where $A$ is a source algebra for the unique block $\cO G$. We have $\Hom(D/\mathfrak{foc}(D),\mathcal{O}^\times) \cong C_D(N_G(D,B_D)) = P$. The result follows from the description of $\cL(B)$ in (\ref{arr:exact}) above, as it is clear that the elements of $\Out_D(A)$ cannot commute with the elements of $\Out(D,\cF)$ obtained as automorphisms of $Q \rtimes C_7$.
\end{proof}

\begin{remark}
In case (\ref{C7Pic}), whilst a Sylow $3$-subgroup of the Picard group for the block in question occurs as a conjugate of $\Out(D,\cF)$, this is not known to be the case for every block Morita equivalent to it. In other words, it is theoretically possible for the Picard group of a Morita equivalent block $C$ to have a subgroup $C_7 \rtimes C_3$, but for $\cT(C) \not\cong C_7 \rtimes C_3$. We will have to beware of this inconvenience in Section \ref{extensions}.
\end{remark}


\section{Preliminaries on perfect isometries}
\label{PI}

We require a method for comparing the principal blocks of $\cO (A_4 \times P)$ and $\cO (A_5 \times P)$ with those of $\cO (A_4 \times Q)$ and $\cO (A_5 \times Q)$ respectively when $Q$ is a subgroup of an abelian $2$-group $P$. We will do this in Section \ref{index2compare}, but first require an analysis of their perfect self-isometries, which is the content of this section. For a block $B$, write ${\Perf} (B)$ for the group of perfect self-isometries of $B$, under composition. The results of this section are an extension of those of part of Section 2 of~\cite{el18}. 

Note that every perfect isometry $I$ between blocks $B_1$ and $B_2$ gives rise to a bijection of character idempotents and so to a $K$-algebra isomorphism between $Z(KB_1)$ and $Z(KB_2)$, and that by~\cite{br90} this induces an $\cO$-algebra isomorphism $\phi_I:Z(B_1) \rightarrow Z(B_2)$.

By \thref{Dabelian_inner:prop}, for a block with abelian defect groups, every irreducible character in a block of a normal subgroup of index $p$ covered by $B$ is $G$-stable and extends to $G$. The following Proposition tells us that these extensions behave well with respect to perfect isometries.

\begin{proposition}
\thlabel{prop:isomcent}
For $i=1,2$ let $G_i$ be a finite group and $N_i \lhd G_i$ with index $p$. Let $B_i$ be a block of $\cO G_i$ with abelian defect group $D$ and let $b_i$ be a $G_i$-stable block of $\cO N_i$ covered by $B_i$. For each $\chi \in \Irr(b_1)$ write $\Irr(B_1,\chi)= \{\chi_1,\ldots, \chi_p \}$.

Suppose $I:\mathbb{Z}\Irr(B_1)\to\mathbb{Z}\Irr(B_2)$ is a perfect isometry such that for each $\chi \in \Irr(b_1)$ there is $\psi \in \Irr(b_2)$ and $\epsilon_\chi \in\{\pm1\}$ such that $I(\chi_i)=\epsilon_\chi \psi_i$ for $i=1, \ldots, p$ where $\Irr(B_2,\psi)= \{\psi_1, \ldots , \psi_p \}$. Then the isometry $I_{N_1,N_2}:\mathbb{Z}\Irr(b_1)\to\mathbb{Z}\Irr(b_2)$ defined by $I_{N_1,N_2}(\chi):=\epsilon_\chi \psi$ is perfect and further $\phi_{I_{N_1,N_2}}=\phi_I|_{Z(b_1)}$.
\end{proposition}

\begin{remark}
We may restrict $\phi_I$ to $Z(b_1)$ since, as by \thref{Dabelian_inner:prop} $B_i$ acts as inner automorphisms on $b_i$ and so $Z(b_i) \subseteq Z(B_i)$.
\end{remark}

\begin{proof}
This is Proposition 2.6 and Lemma 2.7 of~\cite{el18}.
\end{proof}

\begin{lemma}
\thlabel{lem:perfP}
Let $P$ be an abelian finite $p$-group. Then ${\Perf}(\cO P)\cong{\Aut}(\cO P)\times C_2$.
\end{lemma}

\begin{proof}
Since there is only one indecomposable projective module for $\cO P$, every perfect self-isometry of $\cO P$ must have all positive or all negative signs. Now every perfect self-isometry induces a permutation of $\Irr(P)$, which induces an automorphism of $\Aut(Z(\cO P))\cong \Aut(\cO P)$, and the result follows.
\end{proof}

Now consider the character table of $A_4$. Let $\omega$ be a primitive $3$rd root of unity. We set up the labelling of characters:

\begin{align*}
\begin{tabular}{|c||c|c|c|c|} \hline
 & $()$ & $(12)(34)$ & $(123)$ & $(132)$ \\ \hline
$\chi_1$ & $1$ & $1$ & $1$ & $1$ \\
$\chi_2$ & $1$ & $1$ & $\omega$ & $\omega^2$ \\
$\chi_3$ & $1$ & $1$ & $\omega^2$ & $\omega$ \\
$\chi_4$ & $3$ & $-1$ & $0$ & $0$ \\ \hline
\end{tabular}
\end{align*}

For the rest of the section we assume $p=2$.

\begin{proposition}
\thlabel{prop:self_A4}
The perfect self-isometries of $\mathcal{O}A_4$ are precisely the isometries of the form:
\begin{align*}
I_{\sigma,\epsilon}:\mathbb{Z}\Irr(A_4)&\to\mathbb{Z}\Irr(A_4)\\
\chi_j&\mapsto\epsilon\delta_j\delta_{\sigma(j)}\chi_{\sigma(j)}
\end{align*}
for $1\leq j\leq 4$, where $\sigma\in S_4$, $\epsilon\in\{\pm1\}$ and $\delta_1=\delta_2=\delta_3=-\delta_4=1$. Hence ${\Perf}(B_0(\cO A_5) \cong {\Perf}(\cO A_4)\cong C_2 \times S_4$.
\end{proposition}

\begin{proof}
This is~\cite[Proposition 2.8]{el18} together with the observation that by~\cite[A1.3]{br90} $\cO A_4$ and $B_0(\cO A_5)$ are perfectly isometric.
\end{proof}

\begin{theorem}
\thlabel{thm:PA4}
Let $P$ be a finite abelian $2$-group. Every perfect self-isometry of $\mathcal{O}(P\times A_4)$ is of the form $(J,I_{\sigma,\epsilon})$, where $J$ is a perfect isometry of $\cO P$ induced by an $\cO$-algebra automorphism, $\sigma\in S_4$ and $\epsilon\in\{\pm1\}$.
\end{theorem}

\begin{proof}
We proceed as in the proof of~\cite[Theorem 2.11]{el18}. The set of projective indecomposable characters (characters of projective indecomposable modules) is
\begin{align*}
\prj (\mathcal{O}(P\times A_4))=\{\chi_{P_1},\chi_{P_2},\chi_{P_3}\},\text{ where }\chi_{P_j}:=\left(\sum_{\theta\in\Irr(P)}\theta\right)\otimes\left(\chi_j+\chi_4\right).
\end{align*}

Let $I$ be a perfect self-isometry of $\mathcal{O}(P\times A_4)$. Each $I(\chi_{P_i})$ is an integer linear combination of projective indecomposable characters. By counting constituents we see that
\begin{align}\label{align:im}
I(\chi_{P_u})=\pm\chi_{P_1},\pm\chi_{P_2},\pm\chi_{P_3},\pm(\chi_{P_1}-\chi_{P_2}),\pm(\chi_{P_1}-\chi_{P_3})\text{ or }\pm(\chi_{P_2}-\chi_{P_3}),
\end{align}
for $1\leq u\leq 3$. Consider the set
\begin{align*}
X_m:=\left\{j : \left\langle\zeta\otimes\chi_j,I\left(\left(\sum_{\theta\in\Irr(P)}\theta\right)\otimes\chi_m\right)\right\rangle\neq0,\text{ for some }\zeta\in\Irr(P)\right\},
\end{align*}
for $1\leq m\leq 4$. By (\ref{align:im}) we have shown that $|X_m|=1$ or $2$ for every $1\leq m\leq 4$. If $|X_1|=2$, then by considering (\ref{align:im}) for $u=1$ we see that $X_4=X_1$. Similarly by considering $I(\chi_{P_2})$, we get that $X_2=X_4$. This is now a contradiction as then
\begin{align*}
I\left(\left(\sum_{\theta\in\Irr(P)}\theta\right)\otimes\left(\chi_1+\chi_2+\chi_4\right)\right)
\end{align*}
has at most $2|P|$ constituents with non-zero multiplicity. Therefore $|X_1|=1$ and so by considering $I(\chi_{P_1})$ we get that $|X_4|=1$ and then by considering $I(\chi_{P_2})$ and $I(\chi_{P_3})$ we get that $|X_2|=|X_3|=1$. Moreover, $X_1,X_2,X_3,X_4$ must all be disjoint. By composing $I$ with the perfect isometry $(\Id_{\ZZ\Irr(P)},I_{\sigma,1})$, for some appropriately chosen $\sigma\in S_4$, we may assume $X_m=\{m\}$ for all $1\leq m\leq4$. Therefore $I(\chi_{P_u})=\pm\chi_{P_u}$ for $1\leq u\leq3$ and by considering
\begin{align*}
I\left(\left(\sum_{\theta\in\Irr(P)}\theta\right)\otimes\chi_4\right),
\end{align*}
we see that in fact all these signs are the same and we may assume, after possibly composing $I$ with $(\Id_{\ZZ\Irr(P)},I_{1,-1})$, that
\begin{align*}
I\left(\left(\sum_{\theta\in\Irr(P)}\theta\right)\otimes\chi_m\right)=\left(\sum_{\theta\in\Irr(P)}\theta\right)\otimes\chi_m,
\end{align*}
for $1\leq m\leq 4$. Next we note that
\begin{align*}
\frac{1}{12}\theta\otimes(\chi_1+\chi_2+\chi_3+3\chi_4)\in\CF(P\times A_4,\mathcal{O}(P\times A_4),\mathcal{O}),
\end{align*}
for each $\theta\in\Irr(P)$. As $3$ is invertible in $\mathcal{O}$, this implies
\begin{align*}
\theta\otimes\left(\sum_{m=1}^4\delta_m\chi_m\right)\in4\CF(P\times A_4,\mathcal{O}(P\times A_4),\mathcal{O}),
\end{align*}
where $\delta_m$ is defined as in \thref{prop:self_A4}, and so
\begin{align}\label{align:im2}
I\left(\theta\otimes\left(\sum_{m=1}^4\delta_m\chi_m\right)\right)\in4\CF(P\times A_4,\mathcal{O}(P\times A_4),\mathcal{O}),
\end{align}
for each $\theta\in\Irr(P)$. Fixing for now $\theta \in \Irr(P)$, define $\theta_m\otimes\chi_m:=I(\theta\otimes\chi_m)$, for $1\leq m\leq4$. Let $x \in P$. Evaluating (\ref{align:im2}) at $(x,1)$, $(x,(123))$ and $(x,(132))$, gives
\begin{align}
\theta_1(x)+\theta_2(x)+\theta_3(x)+\theta_4(x)&\in4\mathcal{O},\label{zeta1}\\
\theta_1(x)+\omega\theta_2(x)+\omega^2\theta_3(x)&\in4\mathcal{O},\label{zeta2}\\
\theta_1(x)+\omega^2\theta_2(x)+\omega\theta_3(x)&\in4\mathcal{O}\label{zeta3}.
\end{align}
Proceeding as in the proof of~\cite[Theorem 2.11]{el18} we have $\theta_1(x)=\theta_2(x)=\theta_3(x)=\theta_4(x)$ for all $x\in P$. In other words $\theta_1=\theta_2=\theta_3=\theta_4$.
\newline
\newline
We have shown that we may assume $I$ is of the form
\begin{align*}
I(\theta\otimes\chi_m)\mapsto J(\theta)\otimes\chi_m
\end{align*}
for all $\theta\in\Irr(P)$, where $J$ is a permutation of $\Irr(P)$. In particular the $\mathcal{O}$-algebra automorphism of $Z(\mathcal{O}(P\times A_4)$ induced by $I$ leaves $\mathcal{O}P$ invariant. Therefore the permutation $J$ of $\Irr(P)$ must induce an automorphism of $\mathcal{O}P$ and the theorem is proved.
\end{proof}

We need two technical lemmas before we continue. Set
\begin{align*}
A&= k[X,Y,Z]/(X^2,Y^2,Z^2,XY,XZ,YZ) \cong Z(kA_4),\\
A_{(m_1,\dots,m_s)}&= k[X_1,\dots,X_s]/(X_1^{2^{m_1}},\dots,X_s^{2^{m_s}}) \cong k(C_{2^{m_1}} \times \cdots \times C_{2^{m_s}}),
\end{align*}
for $s\in\mathbb{N}$, $(m_1,\dots,m_s)\in\mathbb{N}^s$.

\begin{lemma}\label{lem:pgroup}
Let $s,t\in\mathbb{N}$, $(m_1,\dots,m_s)\in\mathbb{N}^s$, with $m_1 \geq \dots \geq m_s$ and $(n_1,\dots,n_t)\in\mathbb{N}^t$, with $n_1 \geq \dots \geq n_t$. If $A_{(m_1,\dots,m_s)}\otimes_k A \cong A_{(n_1,\dots,n_t)}\otimes_k A$, then $s=t$ and $m_1=n_1, \dots, m_s=n_s$.
\end{lemma}

\begin{proof}
Throughout this proof we use $X_1, \dots, X_s, X, Y, Z$ to denote the images of the elements of the same name in $A_{(m_1,\dots,m_s)}\otimes_k A$.

We first note that
\begin{align*}
\cB:=\Bigg{\{}\Bigg{(}\prod_{i=1}^s X_i^{u_i}\Bigg{)}X^{\epsilon_X}Y^{\epsilon_Y}Z^{\epsilon_Z}|0\leq u_i&<2^{m_i},\text{ for all }1\leq i\leq s,\\
&\epsilon_X,\epsilon_Y,\epsilon_Z\in\{0,1\}, \text{ with } \epsilon_X+\epsilon_Y+\epsilon_Z \leq 1 \Bigg{\}}
\end{align*}
forms a basis for $A_{(n_1,\dots,n_t)}\otimes_k A$ and, setting $J:=J(A_{(m_1,\dots,m_s)}\otimes_k A)$ (the Jacobson radical of $A_{(m_1,\dots,m_s)}\otimes_k A$), we have that
$$
\{X_1+J^2,\dots,X_s+J^2,X+J^2,Y+J^2,Z+J^2\}
$$
forms a basis of $J/J^2$.

Suppose some $W \in J \setminus J^2$ has image
$$
\left(\sum_{i=1}^s\lambda_iX_i\right)+(\lambda_XX+\lambda_YY+\lambda_ZZ)+J^2
$$
in $J/J^2$, for some $\lambda_i,\lambda_X,\lambda_Y,\lambda_Z \in k$. Then $W$ has order of nilpotency at least $2$. If $\lambda_i \neq 0$ for some $1 \leq i \leq s$, then, by looking at the coefficients of powers of $X_i$ with respect to the basis $\cB$, we can see that $W$ has order of nilpotency at least $2^{m_i}$.

Now let $W_1, \dots W_{s+3} \in J$ be such that $\{W_1+J^2, \dots, W_{s+3}+J^2\}$ forms a basis of $J/J^2$. We set $o_i$ to be the order of nilpotency of $W_i$, for $1 \leq i \leq s+3$. By reordering, we may assume that $o_1 \geq \dots \geq o_{s+3}$. As a consequence of the previous paragraph $o_i \geq 2^{m_i}$, for each $1 \leq i \leq s$ and $o_{s+1}, o_{s+2}, o_{s+3} \geq 2$. Moreover, setting $W_i = X_i$, for each $1 \leq i \leq s$ and $W_{s+1}=X, W_{s+2}=Y, W_{s+3}=Z$, these lower bounds on the $o_i$'s can all be achieved. We have now shown that the tuple $(m_1, \dots, m_s)$ can be retrieved from the isomorphism type of $A_{(m_1,\dots,m_s)}\otimes_k A$ and the result is proved.
\end{proof}

\begin{corollary}
\thlabel{lem:notPI}
Let $P$ and $Q$ be finite abelian $2$-groups. If $\cO(P\times A_4)$ is perfectly isometric to $\cO(Q\times A_4)$, 
then $P\cong Q$.
\end{corollary}

\begin{proof}
Suppose $\cO(P\times A_4)$ is perfectly isometric to $\cO(Q\times A_4)$. Then certainly $Z(k(P\times A_4))\cong Z(k(Q\times A_4))$. We may assume that $P,Q \neq 1$. Note that $Z(kA_4)\cong A$, $kP\cong A_{(m_1,\dots,m_s)}$ and $kQ\cong A_{(n_1,\dots,n_t)}$, where $P\cong C_{2^{m_1}}\times\dots\times C_{2^{m_s}}$ and $Q\cong C_{2^{n_1}}\times\dots\times C_{2^{n_t}}$. By Lemma~\ref{lem:pgroup}, we have $\{m_1,\dots,m_s\}=\{n_1,\dots,n_t\}$ and so $P\cong Q$.
\end{proof}

\section{Normal subgroups of index 2}
\label{index2compare}

\begin{proposition}[Theorem 3.15 of~\cite{el19}]\label{prop:grunit}
Let $G$ be a finite group and $N$ a normal subgroup of $G$ of index $p$. Now let $B$ be a block of $\mathcal{O}G$ with abelian defect group $D$ such that $G=ND$. Then there exists a block $b$ of $\mathcal{O}N$ with the same block idempotent as $B$ and defect group $D\cap N$. Moreover there exists a $G/N$-graded unit $a\in Z(B)$, in particular $B=\bigoplus_{j=0}^{p-1}a^jb$.
\end{proposition}

\begin{theorem}
\thlabel{index2theorem}
Let $G$, $N$, $B$, $b$ and $D$ be as in Proposition~\ref{prop:grunit}. Suppose further that $D\cong P\times (C_2)^2$, for some finite abelian $2$-group $P$, $D\cap N\cong Q\times (C_2)^2$, for some subgroup $Q\leq P$ of index $2$, and that $b$ has inertial quotient $C_3$ and is Morita equivalent to the principal block of $\mathcal{O}(Q\times A_4)$ (respectively $\mathcal{O}(Q\times A_5)$). Then $B$ is Morita equivalent to the principal block of $\mathcal{O}(P\times A_4)$ (respectively $\mathcal{O}(P\times A_5)$).
\end{theorem}

\begin{proof}
We follow the proof of \cite[Theorem 2.15]{el18}.

Suppose that $b$ is Morita equivalent to the principal block of $\mathcal{O}(Q\times A_4)$ or $\mathcal{O}(Q\times A_5)$ and has inertial quotient $C_3$. By Proposition \ref{Dabelian_inner:prop} $D$ acts as inner automorphisms on $b$, so every irreducible Brauer character of $b$ is fixed under conjugation in $G$. Since $G/N$ is a cyclic $2$-group, each irreducible Brauer character extends uniquely to $G$ and lies in $B$ (the unique block of $G$ covering $b$), so $l(B)=l(b)=3$. By \thref{index_p_inertial:prop} $B$ and $b$ have isomorphic inertial quotients $(C_3)_1$. Since $l(B)$ is equal to the order of the inertial quotient, by the main result of~\cite{wa05} and its proof we have a perfect isometry 
\begin{align*}
\mathbb{Z}\Irr(B)\to\mathbb{Z}\Irr(\mathcal{O}(P\times A_4)).
\end{align*}
 
Similarly there is a perfect isometry $\mathbb{Z}\Irr(B_0(\cO (P \times A_5)))\to\mathbb{Z}\Irr(\mathcal{O}(P\times A_4))$.

Write $\cA$ for $A_4$ or $A_5$. From the above we have a perfect isometry
\begin{align*}
I:\mathbb{Z}\Irr(B)\to\mathbb{Z}\Irr(B_0(\mathcal{O}(P\times \cA))),
\end{align*}
for $\cA$ as according to the Morita equivalence class of $b$.

Now $I$ induces an isomorphism of groups $\Perf(B)\cong \Perf(B_0(\cO(P\times \cA)))$ via $\beta \mapsto I \circ \beta \circ I^{-1}$ for $\beta$ any perfect self-isometry of $B$, and we denote this isomorphism by $I_{\operatorname{PI}}$. Consider the perfect self-isometry
\begin{align*}
L:\mathbb{Z}\Irr(B)&\to\mathbb{Z}\Irr(B)\\
\chi&\mapsto\operatorname{sgn}_N^G.\chi,
\end{align*}
where $\operatorname{sgn}_N^G$ is the linear character of $G$ with kernel $N$. So for each $\theta\in\Irr(b)$, $L$ swaps the two extensions of $\theta$ to $G$. We know that $L$ is indeed a perfect isometry as it is induced by the $\mathcal{O}$-algebra automorphism of $\mathcal{O}G$ given by $g\mapsto\operatorname{sgn}_N^G(g)g$ for all $g\in G$.

Note that $L$ is a perfect self-isometry of order $2$ and that it induces the trivial $k$-algebra automorphism on $Z(kB)$. Furthermore, since by Proposition \ref{index_p_inertial:prop} induction gives a bijection between $\prj(b)$ and $\prj(B)$, each element of $\prj (B)$ is fixed under multiplication by $\operatorname{sgn}_N^G$ and so $L$ is the identity on $\mathbb{Z}\prj (B)$. Therefore $I_{\operatorname{PI}}(L)$ must be of order $2$, induce the identity $k$-algebra automorphism on $Z(B_0(k(P\times \cA)))$ and be the identity on $\mathbb{Z}\prj (B_0(\mathcal{O}(P\times \cA)))$. We claim that $I_{\operatorname{PI}}(L)$ is induced by multiplication by the sign character of $G':=P\times \cA$ with respect to the subgroup $N':= R\times \cA \leq G'$, for some index $2$ subgroup $R \leq P$, with $R \cong Q$.

We first deal with the $\cA=A_4$ case. Adopting the notation of \thref{thm:PA4}, set $I_{\operatorname{PI}}(L)=(J,I_{\sigma,\epsilon})$, where $J$ is a perfect self-isometry of $\cO P$ induced by an $\cO$-algebra automorphism, $\sigma\in S_4$ and $\epsilon\in\{\pm1\}$. Then the fact that $I_{\operatorname{PI}}(L)$ is the identity on $\mathbb{Z}\prj (\mathcal{O}(P\times A_4))$ forces $\sigma$ to be the identity permutation and $\epsilon=1$. So $J$ is induced by an element of $\alpha\in\Aut(\cO P)$ that has order $2$ and induces the identity on $kP$. Recall that $\Aut(\cO P)\cong\Hom(P,\cO^\times)\rtimes\Aut(P)$.
The fact that $\alpha$ induces the identity on $kP$ forces $\alpha$ to be given by multiplication by $\lambda_\alpha \in\Hom(P,\cO^\times)$ of order $2$. In other words $\alpha$ is induced by multiplication by the sign character of $P$ with respect to some normal subgroup $R$ of index $2$. Hence $I_{\operatorname{PI}}(L)$ is induced by multiplication by the sign character of $P\times A_4$ with respect to the subgroup $N':=R\times A_4\leq G'$. (Note that we do not know yet that $R \cong Q$.)

We have now shown that
\begin{align*}
I(\operatorname{sgn}_N^G.\chi)=\operatorname{sgn}_{N'}^{G'}.I(\chi),
\end{align*}
for all $\chi \in \Irr(B)$. By Proposition~\ref{prop:isomcent}, $b$ is then perfectly isometric to $\cO N'$. However, $b$ is Morita equivalent to $B_0(\cO(Q\times A_4))$ and so \thref{lem:notPI} implies that $R\cong Q$ as desired.

For the $\cA=A_5$ case we fix a perfect isometry $I_A:\ZZ \Irr(B_0(\cO A_5)) \to \ZZ \Irr(\cO A_4)$. As above, we can then show that
$$
(\Id_{\ZZ\Irr(P)},I_A) \circ I_{\operatorname{PI}}(L) \circ (\Id_{\ZZ\Irr(P)},I_A)^{-1}: \ZZ \Irr(\cO(P \times A_4)) \to \ZZ \Irr(\cO(P \times A_4))
$$
is induced by multiplication by the sign character of an appropriate subgroup. That $I_{\operatorname{PI}}(L)$ is of the desired now follows immediately.

Composing the perfect isometry induced by the Morita equivalence between $b$ and $B_0(\cO(Q\times \cA))$ with that given by the isomorphism between $Q \times \cA$ and $N'$, we obtain a perfect isometry $I_{\operatorname{Mor}}: \ZZ \Irr(b) \rightarrow \ZZ \Irr(B_0(\cO N'))$.

Denote by $I_{N,N'}$ the perfect isometry $\ZZ \Irr(b) \rightarrow \ZZ \Irr(B_0(\cO N'))$ induced by $I$ as in Proposition \ref{prop:isomcent}.  

Write $I_{N,N'}\circ I_{\operatorname{Mor}}^{-1}=(J',I_{\tau,\delta})$ in the notation of \thref{thm:PA4} applied to $B_0(\mathcal{O}(R\times \cA))$, where $J'$ is a perfect self-isometry of $\cO R$ induced by an $\cO$-algebra automorphism $\alpha'$, $\tau\in S_4$ and $\delta\in\{\pm1\}$. By post-composing $I$ with the perfect self-isometry $(\Id_{\ZZ\Irr(P)},I_{\tau,\delta})^{-1}$ of $B_0(\cO G')$ and post-composing the Morita equivalence $b\sim_{\operatorname{Mor}}B_0(\cO N')$ with that induced by $\alpha'\otimes\Id_{B_0(\cO A)}$, we may assume that $I_{N,N'}=I_{\operatorname{Mor}}$.

Let $\phi_I:Z(B)\to Z(B_0(\cO G'))$ be the isomorphism of centres induced by $I$ as in Lemma~\ref{prop:isomcent} and let $M$ be the $B_0(\cO N')$-$b$-bimodule inducing the above Morita equivalence $b\sim_{\operatorname{Mor}}B_0(\cO N')$. Since $I_{N,N'}=I_{\operatorname{Mor}}$, by Proposition~\ref{prop:isomcent} we have that $\phi_I|_{Z(b)}=\phi_{I_{N,N'}}:Z(b)\to Z(B_0(\cO N'))$ is the isomorphism of centres induced by the Morita equivalence. In other words,
\begin{align}\label{centre:mor}
\phi_I(\mathsf{b})m=m \mathsf{b},\text{ for all } \mathsf{b}\in b,m\in M.
\end{align}
Let $a\in B$ be a graded unit as described in Proposition~\ref{prop:grunit} and set $a':=\phi_I(a)$. Since $\phi_I$ respects the $G/N$ and $G'/N'$-gradings, $a'$ is also a graded unit. We now give $M$ the structure of a module for
\begin{align*}
(B_0(\cO N')\otimes_{\mathcal{O}}b^{\operatorname{op}})\oplus(a'^{-1}B_0(\cO N')\otimes_{\mathcal{O}}(ab)^{\operatorname{op}})
\end{align*}
by defining $a'^{-1}.m.a=m$, for all $m\in M$, where (\ref{centre:mor}) ensures that this does indeed define a module. Now by~\cite[Theorem 3.4]{mar96} we have proved that $B$ is Morita equivalent to $B_0(\mathcal{O}(P \times \cA))$.

\end{proof}


\section{Extensions of blocks}
\label{extensions}

In this section we give the possible Morita equivalence classes of blocks covering a block of a normal subgroup in some relevant Morita equivalence classes. 

Let $G$ be a finite group and $N \lhd G$. Let $b$ be a $G$-stable block of $\cO N$ covered by a block $B$ of $G$ with abelian defect group $D$. Then $b$ has defect group $Q=D \cap N$. Let $(D,B_D)$ be a maximal $B$-subpair.

We first extract and summarize two results of~\cite{wzz18} and~\cite{zh16}.

The following is a weaker version of Theorem 5.10 of~\cite{wzz18} that is sufficient for our purposes.

\begin{theorem}[\cite{wzz18}]
\thlabel{WZZextendp}
With the notation above, suppose further that $[G:N]$ is a power of $p$, so that $G=DN$. Let $E=N_G(D,B_D)/C_G(D)$ suppose that $E$ is cyclic and acts freely on $[N_G(D,B_D),D] \setminus \{ 1 \}$ (that is, all orbits have length $|E|$). Suppose that $b$ is inertial, i.e., there is a basic Morita equivalence with $\cO (Q \rtimes E)$. Then $B$ is Morita equivalent to $\cO (D \rtimes E)$.
\end{theorem}

The following is the main result of~\cite{zh16}, and is particularly relevant to the case that $b$ is a nilpotent covered block.

\begin{theorem}[\cite{zh16}]
\thlabel{inertialp'extend}
Suppose that $N$ has $p'$-index. If $b$ is inertial, then $B$ is inertial.
\end{theorem}

We now apply K\"ulshammer's analysis in~\cite{ku95} of the situation of a normal subgroup containing the defect groups of a block, which involves the study of crossed products of a basic algebra with an $p'$-group. In the general setting he finds finiteness results for the possible crossed products, but in our situation we are able to precisely describe the possibilities.  

Background on crossed products may be found in~\cite{ku95}, but we summarize what we need here. Let $X$ be a finite group and $R$ an $\cO$-algebra. A crossed product of $R$ with $X$ is an $X$-graded algebra $\L$ with identity component $\L_1 = R$ such that each graded component $\L_x$, where $x \in X$, contains a unit $u_x$. Given a choice of unit $u_x$ for each $x$, we have maps $\alpha: X \rightarrow \Aut (R)$ given by conjugation by $u_x$ and $\mu:X \times X \rightarrow U(R)$ given by $\alpha_x \circ \alpha_y = \iota_{\mu(x,y)} \circ \alpha_{xy}$, where $U(R)$ is the group of units of $R$ and $\iota_{\mu(x,y)}$ is conjugation by $\mu(x,y)$. The pair $(\alpha,\mu)$ is called a parameter set of $X$ in $R$. In~\cite{ku95} an isomorphism of crossed products respecting the grading is called a weak equivalence. By the discussion following Proposition 2 of~\cite{ku95} weak isomorphism classes of crossed products of $R$ with $X$ are in bijection with pairs consisting of an $\Out(R)$-conjugacy class of homomorphisms $X \rightarrow \Out(R)$ for which the induced element in $H^3(X,U(Z(R)))$ vanishes, and an element of $H^2(X,U(Z(R)))$.

We adapt Proposition in Section 3 of~\cite{ku95}, and include a proof for completeness (as given in~\cite{ea19}). Note that $\alpha: X \rightarrow \Aut (R)$ restricts to a map $X \rightarrow \Aut (Z(R))$. Hence we also have homomorphisms $X \rightarrow \Aut(Z(R)/J(Z(R)))$ and $X \rightarrow \Aut(U(Z(R)/J(Z(R))))$.

\begin{lemma}
\thlabel{vanishing_cohomology}
Suppose that $X = \langle x \rangle$ is a cyclic $p'$-group. Then $H^i(X,U(Z(R)))=0$ for each $i \geq 1$.
\end{lemma}

\begin{proof}
Let $i \in \mathbb{N}$. Following~\cite{ku95}, $U(Z(R)) \cong U(Z(R)/J(Z(R))) \times (1+J(Z(R))$ and $H^i(X,U(Z(R))) \cong H^i(X,U(Z(R)/J(Z(R)))) \times H^i(X,1+J(Z(R)))$. We have $H^i(X,1+J(Z(R)))=0$ since $X$ is a $p'$-group. Now $Z(R)/J(Z(R))$ is a commutative semisimple $k$-algebra, which we denote by $A$, and note as above we have a homomorphism $X \rightarrow \Aut(A)$. Write $A=A_1 \times \cdots \times A_r$, where each $A_j$ is a product of simple algebras constituting an $X$-orbit. We have $H^i(X,U(A)) \cong H^i(X,U(A_1)) \times \cdots \times H^i(X,U(A_r))$. Now each $H^i(X,U(A_j))$ vanishes, for as a $kX$-module $A_j$ is induced from the trivial module of $kY$ for some $Y \leq X$, and so by Shapiro's Lemma $H^i(X,U(A_j)) \cong H^i(Y,k^\times)$ (see~\cite[2.8.4]{ben1}), which vanishes since $X$ is cyclic. Hence $H^i(X,U(Z(R)))=0$ for each $i$.
\end{proof}

We apply the above with $X=G/N=\langle x \rangle$, where $G/N$ is a $p'$-group. Let $f$ be an idempotent of $b$ such that $fbf$ is a basic algebra for $b$. By~\cite[Lemma 4.2]{ar21} $fBf$ is a crossed product of $fbf$ with $X$ and $fBf$ is Morita equivalent to $B$. Hence we may take $R=fbf$ in the above. By \thref{vanishing_cohomology} weak isomorphism classes of crossed products of $fbf$ with $X$ are in bijection with $\Out(fbf)$-conjugacy classes of homomorphisms $X \rightarrow \Out(fbf)$. Note however, that such crossed products may be isomorphic as algebras but not weakly isomorphic as crossed products. Indeed, given $\alpha: X \rightarrow \Out(fbf)$, the same algebra gives rise to parameter sets associated to $\alpha \circ \varphi$ for each $\varphi \in \Aut(X)$.

Now $\Out(fbf)$ embeds in $\Pic(fbf) \cong \Pic(b)$ and since $fbf$ is a basic algebra $\Out(fbf) \cong \Pic(fbf)$, so we may apply the descriptions of Picard groups in \thref{pic_list}. The strategy will be to limit the number of possible Morita equivalence classes for $B$ (given $b$), and to identify examples where all such Morita equivalence classes are realised.

A special case that will arise frequently, and that demonstrates the phenomenon of crossed products isomorphic as algebras but not weakly isomorphic as crossed products is as follows:

\begin{lemma}
\thlabel{Picard_argument_special_case}
With the notation above, suppose $G/N$ has prime order $r$ different to $p$ and that $\Out(fbf)$ has cyclic Sylow $r$-subgroups of order $r$. Then there are precisely two possibilities for Morita equivalence class of $B$, one of which is that $B$ is source algebra equivalent to $b$.
\end{lemma}

\begin{proof}
The trivial homomorphism $G/N \rightarrow \Out(fbf)$ corresponds to the case that $B$ is source algebra equivalent to $b$ by \thref{innerautos}, since $G=G[b]$. Consider nontrivial $\alpha: G/N \rightarrow \Out(fbf)$, and consider a block $B$ such that $fBf$ is a crossed product of $fbf$ with $G/N$ corresponding to $\alpha$. Then for each $\varphi \in \Aut(G/N)$ we have that $\alpha \circ \varphi$ also realises $fBf$ as a crossed product. This accounts for all possible homomorphisms $\alpha: G/N \rightarrow \Out(fbf)$.
\end{proof}

\begin{example}
\thlabel{Puig_example}
As observed in~\cite{ar21}, there are examples of nilpotent blocks covering non-nilpotent blocks constructed in~\cite[Remark 4.4]{pu11} that will be useful in the arguments that follow. Let $P$ be an abelian $2$-group on which $C_r$ acts regularly for a prime $r$. Define $N=(P \rtimes C_r) \times C_r$ with Sylow $r$-subgroup $H$. Note that $N$ has $r$ $2$-blocks, corresponding to $\Irr(Z(N))$. There is a group $T \cong C_r$ acting on $N$ fixing the elements of $Z(N)$ and $H/Z(N)$. Define $G=N \rtimes T$. Then each nontrivial element of $\Irr(Z(N))$ is covered by a nilpotent block of $G$ (with defect group $P$). These cover block of $N$ Morita equivalent to $\cO (P \rtimes C_r)$. We in particular require the cases $N=G_n=(C_{2^n})^2 \rtimes C_3$ and $N=(C_{2^n})^3 \rtimes C_7$.
\end{example}

\begin{remark}
\thlabel{conj_perm}
In the crossed product construction, we are considering homomorphisms $G/N \rightarrow \Pic(b)$ induced by congugation by elements of $G$. Conjugation on $\cO N$ affords a permutation module, and so conjugation on an invariant summand (in our case $b$) affords a trivial source module. It follows that actually we have $G/N \rightarrow \cT(b)$. In general this is a less useful observation than we might hope, as $\cT(b)$ is not (known to be) preserved under Morita equivalence. However, it does allow us to keep some control over inertial quotients as we move from $b$ to $B$. We will use this observation in the following proofs.    
\end{remark}

We now apply the above to every extension of a Morita equivalence from a normal subgroup of odd prime index that we will need.

\begin{proposition}
\thlabel{Picard_application}
Let $G$ be a finite group and $N \lhd G$ with $[G:N]=r$, where $r$ is an odd prime. Let $B$ be a $2$-block of $\cO G$ covering a $G$-stable block of $b$ of $\cO N$.

\begin{enumerate}[(i)]
    \item\label{DC_7} Suppose that $b$ is Morita equivalent to $\cO(((C_{2^{n}})^3 \rtimes C_7) \times P)$ where $n \in \NN$ and $P$ is a cyclic $2$-group. Suppose $r=3$. Then $B$ is either source algebra equivalent to $b$ or Morita equivalent to $\cO(((C_{2^{n}})^3 \rtimes (C_7 \rtimes C_3)) \times P)$. Further, if $b$ is known to have inertial quotient $C_7$, then in the latter case $B$ has inertial quotient $C_7 \rtimes C_3$. Suppose $r=7$. Then $B$ is either source algebra Morita equivalent to $b$ or is nilpotent. If $r$ is an odd prime other than $3$ and $7$, then $B$ is source algebra equivalent to $b$.

    \item\label{DC_3} Suppose that $b$ is Morita equivalent to $\cO(G_m \times P)$, where $m \in \NN$ and $P$ is an abelian $2$-group with rank at most $2$. Let $D$ be the defect group of $b$. If $r \neq 3$, then $B$ is source algebra equivalent to $b$. If $r=3$ and $P \cong (C_{2^n})$ for some $n \in \NN$, then $B$ is either nilpotent, source algebra equivalent to $b$, or Morita equivalent to $\cO(G_m \times G_n)$ or a nonprincipal block of  $\cO(D \rtimes 3_+^{1+2})$. Further, if $b$ is known to have inertial quotient $(C_3)_1$, then in the latter two cases the inertial quotient of $B$ is $C_3 \times C_3$. If $r=3$ and $P$ is cyclic or a product of two cyclic groups of different orders, then $B$ is either nilpotent or source algebra equivalent to $b$. 
    \item\label{A_5xP} Suppose that $b$ is Morita equivalent to $B_0(\cO (A_5 \times P))$ where $P$ is an abelian $2$-group with rank at most $2$. If $P$ is cyclic or a product of two cyclic groups of different orders, then $B$ is source algebra equivalent to $b$. If $P \cong (C_{2^n})$ for some $n \in \NN$, then $B$ is either source algebra equivalent to $b$ or $B_0(\cO (A_5 \times G_n))$, with the latter case only occurring when $r=3$.
    \item\label{A_5xG_n} Suppose that $b$ is Morita equivalent to $B_0(\cO (A_5 \times G_n))$. Then $B$ is Morita equivalent to $B_0(\cO (A_5 \times (C_{2^{n}})^2))$ or source algebra equivalent to $b$, with the former case only occurring when $r=3$.

\end{enumerate}

\end{proposition}

\begin{proof}
 (\ref{DC_7}) By \thref{pic_list} $\Pic(b) \cong (C_7 \rtimes C_3) \times (C_{2^n} \rtimes \Aut(C_{2^n}))$. By \thref{Picard_argument_special_case}, for both $r=3$ and $r=7$ there are two possibilities for the Morita equivalence class of $B$, and one is that $B$ is source algebra equivalent to $b$. For $r=7$ the second is that $B$ is nilpotent as realised in \thref{Puig_example}. For $r=3$, the second case is realised by the group given in the statement. For all other odd primes, since they do not divide the order of the Picard group, $B$ is source algebra equivalent to $b$. It remains to prove the statement regarding inertial quotients. Suppose that $b$ is known to have inertial quotient $C_7$ and that $r=3$. It follows from \thref{inertial_quotient_normal} that $B$ has inertial quotient $C_7$ or $C_7 \rtimes C_3$. By \thref{conj_perm} we observe that we may consider elements of $\cT(b)$. By the description of $\cT(b)$ in Section \ref{Picard} we know that $\cT(b)$ maps to a subgroup of $\Out(D,\cF) \cong C_3$. If this subgroup is trivial, then we may not construct any nontrivial homomorphism $G/N \rightarrow \cT(b)$ and we are done. Hence suppose the subgroup has order three. Then the action of $G$ induces an element of order three in $\Aut(D,\cF)$, so that $B$ must have inertial quotient $C_7 \rtimes C_3$.

 (\ref{DC_3}) By \thref{pic_list} $\Pic(b) \cong S_3 \times (P \rtimes \Aut(P))$. If $P$ is cyclic or a product of two cyclic groups of different orders, then $\Aut(P)$ is a $2$-group, and so by \thref{Picard_argument_special_case} $B$ is either Morita equivalent to $b$ or is nilpotent, with this case realised as in \thref{Puig_example}. Suppose that $P \cong (C_{2^n})$ for some $n \in \NN$. Then $\Pic(b) \cong S_3 \times (P \rtimes S_3)$. There are four conjugacy classes of homomorphisms $G/N \rightarrow \Pic(b)$, and so at most four possibilities for the Morita equivalence class of $B$. These are account for by the four cases listed in the statement, with the cases that $B$ is Morita equivalent to $b$ or $\cO(G_m \times G_n)$ requiring no further explanation. The case that $B$ is nilpotent is again realised as in \thref{Puig_example}. Note that $D \rtimes 3_+^{1+2}$ has a normal subgroup $M$ of index $3$ isomorphic to $N \times C_3$. Now $B$ covers a nonprincipal block of $\cO M$, which is Morita equivalent to $\cO N$. Hence the final case is realised. It remains to prove the statement regarding inertial quotients. Suppose that $b$ is known to have inertial quotient $(C_3)_1$. It follows from \thref{inertial_quotient_normal} that $B$ has inertial quotient $(C_3)_1$ or $C_3 \times C_3$. as above, by \thref{conj_perm} we observe that we may consider elements of $\cT(b)$. By the description of $\cT(b)$ in Section \ref{Picard} we know that $\cT(b)$ maps to a subgroup of $\Out(D,\cF) \cong C_3$. If this subgroup is trivial, then we may not construct any nontrivial homomorphism $G/N \rightarrow \cT(b)$ and we are done. Hence suppose the subgroup has order three. Then the action of $G$ induces an element of order three in $\Aut(D,\cF)$, so that $B$ must have inertial quotient $C_3 \rtimes C_3$.
 
(\ref{A_5xP}) By \thref{pic_list} $\Pic(b) \cong C_2 \times (P \rtimes \Aut(P))$. If $P$ is cyclic or a product of two cyclic groups of different orders, then $\Pic(b)$ is a $2$-group, and so $B$ is Morita equivalent to $b$. Suppose that $P \cong (C_{2^n})$ for some $n \in \NN$. Then $\Pic(b) \cong C_2 \times (P \rtimes S_3)$. Hence by \thref{Picard_argument_special_case} the result follows.

(\ref{A_5xG_n}) By \thref{pic_list} $\Pic(b) \cong C_2 \times S_3$. By \thref{Picard_argument_special_case} there are two possibilities for the Morita equivalence class of $B$, one of which is the class containing $b$. The case $B_0(\cO (A_5 \times (C_{2^{n}})^2))$ is realised by taking a product of $A_5$ with a group as in \thref{Puig_example}.

\end{proof}

\begin{remark}
Since (a) source algebra equivalence preserves the inertial quotient, and (b) for abelian defect groups a block is nilpotent if and only if the inertial quotient is trivial, we have shown that if the inertial quotient of $b$ is isomorphic to that of the given Morita equivalence class representative, then $B$ also has inertial quotient isomorphic to that of the given Morita equivalence class representative.
\end{remark}

Putting the reduction techniques of the section together, we have the following that is in a form directly applicable to the proof of \thref{main_theorem}. Note that we do not need to consider all forms for $b$ here: only those that will arise later.

\begin{proposition}
\thlabel{covers_inertial}
Let $G$ be a finite group and $N \lhd G$ with $G/N$ solvable. Let $B$ be a quasiprimitive $2$-block of $\cO G$ with abelian defect group $D$ of rank at most $4$. Suppose that $B$ covers a block $b$ of $\cO N$.

\begin{enumerate}[(i)]
    \item\label{C7} If $b$ is inertial with inertial quotient $C_7$, then $B$ is Morita equivalent to $\cO D$, $\cO (D \rtimes C_7)$, $\cO (D \rtimes (C_7 \rtimes C_3))$ or $\cO (G_n \times P)$ for some $n$ and some abelian $2$-group $P$, and $B$ has inertial quotient $1$, $C_7$, $C_7 \rtimes C_3$ or $(C_3)_1$ respectively.

    \item\label{C3xC3} If $b$ is inertial with inertial quotient $C_3 \times C_3$, then $B$ is inertial with inertial quotient $1$, $(C_3)_1$ or $C_3 \times C_3$.

    \item\label{C3_2C5C15} If $b$ is inertial with inertial quotient $(C_3)_2$, $C_5$ or $C_{15}$, then $B$ is inertial with inertial quotient $1$, $(C_3)_2$, $C_5$ or $C_{15}$.
    
    \item\label{C3_1} If $b$ is Morita equivalent to a block with normal defect group and has inertial quotient $(C_3)_1$, then $B$ is Morita equivalent to one of:
    \begin{enumerate}[(a)]
    \item $\cO D$;
    \item $\cO (G_{n} \times P)$, where $D \cong (C_{2^{n}})^2 \times P$, and $B$ has inertial quotient $(C_3)_1$;
    \item $\cO (G_{n_1} \times G_{n_2})$, where $D \cong (C_{2^{n_1}})^2 \times (C_{2^{n_2}})^2$, and $B$ has inertial quotient $C_3 \times C_3$;
    \item a non-principal block of $\cO(D \rtimes 3_+^{1+2})$ and $B$ has inertial quotient $C_3 \times C_3$;
    \end{enumerate}
    \item\label{A5factor} If $b$ is Morita equivalent to $B_0(\cO (A_5 \times Q))$ for some $Q \leq D$ and has inertial quotient of type $(C_3)_1$, then $B$ is Morita equivalent to one of:
    \begin{enumerate}[(a)]
     \item $B_0(\cO (A_5 \times P))$, where $D \cong (C_2)^2 \times P$, and $B$ has inertial quotient $(C_3)_1$;
    \item $B_0(\cO (A_5 \times G_n))$, where $D \cong (C_2)^2 \times (C_{2^{n}})^2$, and $B$ has inertial quotient $C_3 \times C_3$.
    \end{enumerate}
    \item\label{C3xC3incA5} If $b$ is Morita equivalent to $B_0(\cO (A_5 \times G_n))$ for some $n \geq 1$, then $B$ is Morita equivalent to $b$ or $B_0(\cO (A_5 \times (C_{2^{n}})^2))$, where $D \cong (C_2)^2 \times (C_{2^{n}})^2$, and $B$ has inertial quotient $C_3 \times C_3$ or $(C_3)_1$ respectively.
     
\end{enumerate}

\end{proposition}

\begin{proof}
Since $b$ is $G$-stable and $G/N$ is solvable, it follows from Lemma \ref{solvablequotient:lem} that $DN/N$ is an abelian Sylow $2$-subgroup of $G/N$, which must then have $2$-length at most one. By Proposition \ref{Dabelian_inner:prop} $D \leq G[b]$, which must also have $2$-length at most one. Hence there are normal subgroups $N_i$ of $G$ such that $N=N_0 \lhd N_1 \lhd N_2 \lhd N_3=G[b]$ with $N_1/N_0$, $N_3/N_2$ of odd order and $N_2/N_1$ a $2$-group, so in particular $N_2 = N_1D$. Write $b_{i}$ for the unique block of $\cO N_i$ covered by $B$. Let $(D,(b_2)_D)$ be a $b_2$-subpair. By \thref{defect_group_factor:prop} $D \cong [D,N_{N_2}(D,(b_2)_D)] \times C_D(N_{N_2}(D,(b_2)_D))$ with $[D,N_{N_2}(D,(b_2)_D)] \leq N_1$ (and so $[D,N_{N_2}(D,(b_2)_D)] \leq N$).

By \thref{innerautos} $b$ is source algebra equivalent to $b_1$, with the same inertial quotient as $b$. If furthermore $b$ is inertial, then $b_1$ is also inertial. 

By \thref{index_p_inertial:prop} $b_2$ has inertial quotient isomorphic to that of $b$, with the same action (this last statement uses the fact that in this case there is a unique action given the isomorphism type of the inertial quotient and the order of commutator part $[D,N_{N_2}(D,(b_2)_D)]$ of $D$).

(\ref{C7}) Suppose that $b$, and so $b_2$, has inertial quotient $C_7$ and that $b$ is inertial. By \thref{WZZextendp} $b_2$ is Morita equivalent (but not necessarily basic Morita equivalent) to $\cO (D \rtimes C_7)$. Note that $C_D(N_{N_2}(D,(b_2)_D))$ is cyclic. By \thref{pic_list} $\Pic(b_2) \cong (C_7 \rtimes C_3) \times (C_D(N_{N_2}(D,(b_2)_D)) \rtimes \Aut(C_D(N_{N_2}(D,(b_2)_D)))$. Now consider $G[b_2]$. Again applying \thref{innerautos}, $b_2$ is source algebra equivalent to the unique block $c$ of $G[b_2]$ covered by $B$. Since $G/G[b_2]$ is an odd order subgroup of $\Pic(b_2)$, it is isomorphic to a subgroup of $C_7 \rtimes C_3$. Suppose that $G/G[b_2] \cong C_7$. By \thref{Picard_application} either $B$ is nilpotent or it is Morita equivalent to $c$ with inertial quotient $C_7$. If $G/G[b_2] \cong C_3$, then again by \thref{Picard_application} $B$ is Morita equivalent to $c$ or $\cO (D \rtimes (C_7 \rtimes C_3))$, with the appropriate inertial quotient. Suppose $G/G[b_2] \cong C_7 \rtimes C_3$. Write $G_1$ for the preimage of $O_7(G/G[b_2])$ in $G$ and $B_1$ for the unique block of $G_1$ covered by $B$. Then by \thref{Picard_application} $B_1$ is either nilpotent or Morita equivalent to $c$. Applying \thref{Picard_application} again, in the first case $B$ is either nilpotent or Morita equivalent to a block with normal defect group and inertial quotient $(C_3)_1$. In the second, $B$ is either Morita equivalent to $c$ or to $\cO (D \rtimes (C_7 \rtimes C_3))$. In each of these cases by \thref{Picard_application} the inertial quotient is as stated.

(\ref{C3xC3}) and (\ref{C3_2C5C15}) Suppose that $b$, and so $b_2$, has inertial quotient $(C_3)_2$, $C_5$, $C_{15}$ or $C_3 \times C_3$, and that $b$ is inertial. We have $C_D(N_{N_2}(D,(b_2)_D))=1$, so $D = [D,N_{N_2}(D,(b_2)_D)] \leq N_1$, i.e., $N_1=N_2$ and $G/N$ has odd order. The result then follows from \thref{inertialp'extend}.

(\ref{C3_1}) Suppose that $b$, and so $b_1$, $b_2$, has inertial quotient $(C_3)_1$. Hence $b_1$ is Morita equivalent to $\cO (G_{n} \times Q)$ for some $n$ and $Q \leq D$. Recalling that $[D,N_{N_2}(D,(b_2)_D)] \leq N_1$, by \thref{index2theorem} $b_2$ is Morita equivalent to $\cO (G_{n} \times Q)$, where $D \cong (C_{2^{n}})^2 \times P$. Applying \thref{innerautos} $b_2$ is source algebra equivalent to the unique block $c$ of $G[b_2]$ covered by $B$. By \thref{pic_list}  $\Pic(b_2) \cong S_3 \times (P \rtimes \Aut(P))$. Suppose that $P$ is not homocyclic. Since $G/G[b_2]$ is an odd order subgroup of $\Pic(b_2)$ and $|\Pic(b_2)|_{2'}=3$, it is isomorphic to a subgroup of $C_3$. By \thref{Picard_application} $B$ is either Morita equivalent to $c$ or is nilpotent. Suppose that $P$ is homocyclic. Then $\Pic(b_2) \cong S_3 \times (P \rtimes S_3)$ and $G/G[b_2]$ is isomorphic to a subgroup of $C_3 \times C_3$. Let $H$ be the kernel of the action of $G$ on the irreducible Brauer characters of $c$ and let $B_H$ be the unique block of $\cO H$ covered by $B$. By \thref{pic_list} $[G:H]$ divides $3$. Suppose that $[G:G[b_2]]=3$. By \thref{Picard_application} the possible Morita equivalence classes for $B$ are represented by the four blocks listed. Suppose that $[G:G[b_2]]=9$. If $G=G[c]$, then by \thref{innerautos} $B$ is Morita equivalent to $c$ and we are done. If $[G:G[c]]=3$, then $B_1$ is Morita equivalent to $c$, and again there are four possibilities for the Morita equivalence class of $B$ and these are as listed. Hence suppose that $G[c]=G[b_2]$, so that $G[b_2]<H<G$. The subgroup of $Pic(c)$ of self-equivalences that preserve all irreducible Brauer characters is isomorphic to $S_3$. Hence there is one possibility for the Morita equivalence class for $B_H$ (as we are excluding the case that $H$ acts as inner automorphisms), that $B_1$ is Morita equivalent to $\cO (G_{n_1} \times G_{n_2})$. One may check as in (\ref{C7}) that the inertial quotients are as stated. 

(\ref{A5factor}) Suppose that $b$, and so $b_1$, $b_2$, is Morita equivalent to $B_0(\cO (A_5 \times Q))$ for some $Q \leq D$. By \thref{index2theorem} $b_2$ is Morita equivalent to $B_0(\cO (A_5 \times P))$, where $D \cong (C_2)^2 \times P$. By \thref{innerautos} $b_2$ is source algebra equivalent to the unique block of $G[b_2]$ covered by $B$. By \thref{pic_list}  $\Pic(b_2) \cong C_2 \times (P \rtimes \Aut(P))$. Since $G/G[b_2]$ is an odd order subgroup of $\Pic(b_2)$ and $P$ has rank at most $2$, it is isomorphic to a subgroup of $C_3$. If $P$ is not homocyclic, then $G=G[b_2]$ and we are done. Suppose $P$ is homocyclic, i.e., $P \cong (C_{2^n})^2$ for some $n$. We may suppose $[G:G[b_2]]=3$. The result then follows by \thref{Picard_application}.
 
(\ref{C3xC3incA5}) Suppose that $b$ is Morita equivalent to $B_0(\cO (A_5 \times G_n))$ for some $n \geq 1$. As in (\ref{C3xC3}) $G/N$ has odd order, so $N_1=N_2$ and $b$ is source algebra equivalent to $b_3$. By \thref{pic_list} $\Pic(b) \cong C_2 \times S_3$. Since $G/G[b]$ is an odd order subgroup of $\Pic(b)$, it is isomorphic to a subgroup of $C_3$. If $G[b]=G$, then $b$ is source algebra equivalent to $B$ and we are done. If $[G:G[b]]=3$, then we are done by \thref{Picard_application}.

\end{proof}

\begin{remark}
We have not treated the case that $b$ is Morita equivalent to $B_0(\cO (A_5 \times A_5)$ here since we do not at present know the Picard group of this block. This case will be treated in the main part of the reduction, where we will make use of additional hypotheses.
\end{remark}

Finally, the methods above may also be used to prove the following.

\begin{proposition}
\thlabel{faithfulblocks}
Let $D \cong (C_{2^{n_1}})^2 \times (C_{2^{n_2}})^2$ for some $n_1, n_2 \in \NN$ and consider $G=D \rtimes 3_+^{1+2}$, where the centre of $3_+^{1+2}$ acts trivially. The $2$-blocks of $\cO G$ correspond to the simple modules of $Z(3_+^{1+2})$, and the two non-principal blocks are Morita equivalent. Further, these blocks are Morita equivalent to the two non-principal blocks of $\cO (D \rtimes 3_-^{1+2})$.
\end{proposition}

\begin{proof}
Let $B$ be any faithful $2$-block of $G=D \rtimes 3_+^{1+2}$ or $D \rtimes 3_-^{1+2}$. Then $l(B)=1$. Take a maximal subgroup $N$ of $G$ and a block $b$ of $N$ covered by $B$. Then $N \cong (D \rtimes C_3) \times C_3$ or $D \rtimes C_9$ and without loss of generality $b$ is Morita equivalent to $\cO (G_{n_1} \times (C_{2^{n_2}})^2)$ for some $n_1, n_2$. By \thref{Picard_application}(\ref{DC_3}) there is only one possibility for the Morita equivalence class of $B$ under the restriction that there is just one simple module.
\end{proof}


\section{Proof of the main theorem}
\label{main_proof}

We first recall the case where the defect group is normal.

\begin{lemma}
\thlabel{normaldefect}
Let $B$ be a block of $\cO G$ for a finite group $G$ with abelian normal defect group $D$ with rank at most $4$ and inertial quotient $E$. Then $B$ is source algebra equivalent to a block of $D \rtimes \hat{E}$, where $E$ is the inertial quotient of $B$ and $\hat{E}/Z(\hat{E})=E$. If $E \cong C_3 \times C_3$, then $\hat{E} \cong E$ or $3^{1+2}$ with $Z(3^{1+2})$ acting trivially on $D$. Otherwise $\hat{E}=E$.
\end{lemma}

\begin{proof}
See~\cite[Theorem 6.14.1]{lin2}. In all cases except $E \cong C_3 \times C_3$, all Sylow subgroups of $E$ are cyclic, so the Schur multiplier is trivial. For $C_3 \times C_3$ the Schur multiplier is $3$ and the result follows.
\end{proof}

We now state a result used in previous reductions for results concerning Morita equivalence classes of blocks, that encapsulates the use of Fong-Reynolds reductions and~\cite{kp90}. It appears in~\cite{ae23} in full, but was extracted from the first part of the proof of~\cite[Proposition 4.3]{eel20}.

\begin{lemma}
\thlabel{reduced_block}
Let $G$ be a finite group and $B$ a block of $\cO G$ with defect group $D$. Then there is a finite group $H$ and a block $C$ of $\cO H$ such that $B$ is basic Morita equivalent to $C$, a defect group $D_H$ of $C$ is isomorphic to $D$ such that:
\begin{enumerate} 
\item[(R1)] $C$ is quasiprimitive, that is, if $N \lhd H$, then $C$ covers a unique block of $\cO N$;
\item[(R2)] If $N \lhd H$ and $C$ covers a nilpotent block of $\cO N$, then $N \leq O_p(H)Z(H)$ with $O_{p'}(N) \leq [H,H]$ cyclic. In particular $O_{p'}(G) \leq Z(G)$;

\item[(R3)] $[H:O_{p'}(Z(H))] \leq [G:O_{p'}(Z(G))]$.

Note that $B$ and $C$ have the same Frobenius categories.
\end{enumerate}
\end{lemma}

\begin{proof}
This is~\cite[Proposition 6.1]{ae23}, in which Fong-Reynolds and K\"ulshammer-Puig reductions are applied repeatedly. It is noted in the proof of~\cite[Proposition 6.1]{ae23} that application of these reductions reduces $[G:O_{p'}(Z(G))]$, hence (R3).
\end{proof}

We call the pair $(H,C)$, where $C$ is a block of $\cO H$, \emph{reduced} if it satisfies conditions (R1) and (R2) \thref{reduced_block}. If the group is clear, then we just say $C$ is reduced.

Before proceeding we recall the definition and some properties of the generalized Fitting subgroup $F^*(G)$ of a finite group $G$. Details may be found in~\cite{asc00}. A \emph{component} of $G$ is a subnormal quasisimple subgroup of $G$. The components of $G$ commute, and we define the \emph{layer} $E(G)$ of $G$ to be the normal subgroup of $G$ generated by the components. It is a central product of the components. The \emph{Fitting subgroup} $F(G)$ is the largest nilpotent normal subgroup of $G$, and this is the direct product of $O_r(G)$ for all primes $r$ dividing $|G|$. The \emph{generalized Fitting subgroup} $F^*(G)$ is $E(G)F(G)$. A crucial property of $F^*(G)$ is that $C_G(F^*(G)) \leq F^*(G)$, so in particular $G/F^*(G)$ may be viewed as a subgroup of $\Out(F^*(G))$. 

\bigskip

\textbf{Proof of \thref{main_theorem}}
Let $B$ be a $2$-block of $\cO G$ for a finite group $G$ with defect group $D$ of rank $4$ with $[G:O_{2'}(Z(G))]$ minimised such that $B$ is not Morita equivalent to any of the blocks listed in the statement of the theorem, or that the inertial quotient of $B$ is not as stated. In the remainder of the proof, the reader may check that the inertial quotients are respected at each step. By \thref{reduced_block} we may assume that $(G,B)$ is reduced. 

If $D \lhd G$ and $B$ has inertial quotient $E_B$, then by the main result of~\cite{ku85} $B$ is source algebra equivalent to either: (i) $D \rtimes E_B$ when $E_B$ has cyclic Sylow $r$-subgroups for every $r$, that is, for every inertial quotient other than $C_3 \times C_3$; a block of $D \rtimes 3_+^{1+2}$ or $D \rtimes 3_-^{1+2}$ when $E_B \cong C_3 \times C_3$, where $Z(3^{1+2})$ acts trivially on $D$. By \thref{faithfulblocks} the non-principal blocks of $D \rtimes 3_+^{1+2}$ and $D \rtimes 3_-^{1+2}$ form a single Morita equivalence class. Hence $B$ is Morita equivalent to a block listed in the statement of the theorem, a contradiction. Hence $D$ is not normal in $G$.

Let $b^*$ be the unique block of $\cO F^*(G)$ covered by $B$.

Write $L_1,\ldots,L_t$ for the components of $G$, so $E(G)=L_1\cdots L_t \lhd G$. Note that $G$ permutes the $L_i$. There must be at least one component, since otherwise $b^*$ is nilpotent and so $F^*(G)=Z(G)O_2(G)$. But $O_2(G) \leq D$ and $D$ is abelian, so we would have $D \leq C_G(F^*(G)) \leq F^*(G)=Z(G)O_2(G)$, so that $D \lhd G$, a contradiction.

Write $b_E$ for the unique block of $E(G)$ covered by $B$ and $b_i$ for the unique block of $L_i$ covered by $b_E$. We claim that no $b_i$ can be nilpotent (in our minimal counterexample).

Let $Z=O_2(Z(E(G)))$. Write $\bar{b}_E$ for the unique block of $\overline{E(G)}:=E(G)/Z$ corresponding to $b_E$ and $\bar{b}_i$ for the unique block of $\bar{L}_i$ corresponding to $b_i$. Writing $M:=\overline{L}_1 \times \cdots \times \overline{L}_t$, where $\overline{L}_i:=L_iZ/Z$, there is a $2'$-group $W \leq Z(M)$ and a block $b_M$ of $M$ with $W$ in its kernel such that $\overline{E(G)} = M/W$ and $b_M$ is isomorphic to $\bar{b}_E$. Then $D \cap E(G)$ is a defect group for $b_E$, $(D \cap E(G))/Z$ is a defect group for $\bar{b}_E$ and $b_M$ has defect groups isomorphic to $(D \cap E(G))/Z$. Then $\bar{b}_i$ has defect group $D_i= ((D \cap E(G))Z) \cap \bar{L}_i$. We have that $b_M=\bar{b}_1 \otimes \cdots \otimes \bar{b}_t$ and $b_M$ has defect group $D_1 \times \cdots \times D_t$. Suppose $b_j$ is nilpotent for some $j$. Let $J \subseteq \{1,\ldots,t\}$ correspond to the orbit of $L_j$ under the permutation action of $G$ on the components. Define $L_J \lhd G$ to be the product of the $L_i$ for $i \in J$, and write $b_J$ for the unique block of $L_J$ covered by $b_E$. For $i \in J$, since $b_j$ is nilpotent, so is $b_i$, hence so is $\bar{b}_i$. Hence the unique block $\bar{b}_J$ of $L_J/Z$ corresponding to $b_J$ is also nilpotent (to see this, observe that as above $L_J/Z$ is the quotient of $\bigtimes_{i \in J} \bar{L}_i$ by a central $2'$-group and that products of nilpotent blocks are nilpotent). It follows that $b_J$ is nilpotent by~\cite{wa94} (where the result is stated over $k$, but follows over $\cO$ immediately), contradicting that $B$ is reduced since $L_J$ is not contained in $O_2(G)Z(G)$. Hence no $b_i$ is nilpotent. In particular, it follows that no $D_i$ can have rank one, and so $t \leq 2$.

Now it follows from Schreier's conjecture that $G/E(G)$ is solvable. By Lemma \ref{solvablequotient:lem} $DE(G)/E(G)$ is a Sylow $2$-subgroup of $G/E(G)$. Since $DE(G)/E(G)$ is abelian, it follows that $G/E(G)$ has $2$-length at most one, so there are normal subgroups $N_i$ of $G$ such that $E(G)=N_0 \lhd N_1 \lhd N_2 \lhd N_3=G[b_E]$ with $N_1/N_0$, $N_3/N_2$ of odd order and $N_2/N_1$ a $2$-group. Further $D \leq N_2$ by Proposition \ref{Dabelian_inner:prop}. Write $b_{N_i}$ for the unique block of $\cO N_i$ covered by $B$.

Suppose that $t=2$. Write $\bar{D}_i:=D_i/O_2(Z(L_i))$, which must have rank $2$ for each $i$. Note that $O_2(G) \leq Z(E(G))$, otherwise a quotient of $D$ would have rank greater than $4$. Further $\bar{D}_i \cong (C_{2^{n_i}})^2$ for some $n_i$, otherwise $\bar{b}_i$, and so $b_i$, would be nilpotent. Now $D_i \cong (C_{2^{n_i}})^2 \times O_2(Z(L_i))$, again since otherwise $b_i$ would be nilpotent. It follows that $O_2(Z(L_i))=1$ for each $i$, hence $O_2(G)=1$. Further note that each $L_i$ is normal in $G$, since $N_G(L_i)$ has index at most $2$ and $G=DN_G(L_i)$. If $G \neq N_G(L_i)$, then some element of $D$ must permute $L_1$ and $L_2$, contradicting our hypothesis that $D$ is abelian. It follows from~\cite[Theorem 1.1]{ekks14} and~\cite{li94} that $b_i$ is Morita equivalent to $\cO G_n$ for some $n \geq 1$ or to $B_0(\cO A_5)$. Further, $b_E$ is isomorphic (in fact basic Morita equivalent) to a block $b_1 \otimes b_2$ of $L_1 \times L_2$.

Suppose $b_i$ is Morita equivalent to $B_0(\cO A_5)$ for each $i$. Then by~\cite[Theorem 1.5]{bkl20} $\Pic(b_i) \cong C_2$, and so $[G:G[b_i]]$ divides $2$. But then $G=DG[b_i]=G[b_i]$. It follows that $G=G[b_E]$. Since the inertial quotient of $b_E$ is $C_3 \times C_3$, it follows that $D=[D,N_{N_2}(D,(b_{N_2})_D)]$, which is contained in $N_1$ by Proposition \ref{defect_group_factor:prop}. We have shown that $[G:E(G)]$ is odd, hence by \thref{innerautos} $B$ is Morita equivalent to $b_E$, and so to $B_0(\cO (A_5 \times A_5))$, a contradiction.

Suppose that $b_1$ is Morita equivalent to $B_0(\cO A_5)$ and $b_2$ to $\cO G_{2^n}$ for some $n$. Then by \thref{covers_inertial}(\ref{C3xC3incA5}) $B$ is Morita equivalent to a block in our list, a contradiction.

Suppose that $b_1$ and $b_2$ are Morita equivalent to $\cO G_{2^{n_1}}$ and $\cO G_{2^{n_2}}$ for some $n_1, n_2 \in \NN$. Then by \thref{Gn_inertial} $b_1$ and $b_2$ are both inertial. Hence $b_E$ is inertial, and so by \thref{covers_inertial} $B$ is inertial, a contradiction.

We have ruled out $t=2$, so now suppose that $t=1$, so that $E(G)=L_1$ is quasisimple. We refer to \thref{qsclassification} for the possibilities for $L_1$ and $b_1$. For ease of notation, write $N=E(G)$ and $b=b_E$. Note that $F^*(G) \cong O_2(G)O_{2'}(G)N$. Here we have already established that $O_{2'}(G) \leq Z(G)$, however it is not immediately the case that $O_{2'}(G) \leq N$, since in principal $G/O_{2'}(G)$ may have Schur multiplier larger than that of $N/Z(N)$. Since $F^*(G)$ is self-centralizing, $G/F^*(G)$ is isomorphic to a subgroup of $\Out(F^*(G))$.

We first consider the cases $N/Z(N) \cong SL_2(8)$, $SL_2(16)$, $J_1$, $Co_3$ and ${}^2G_2(3^{2m+1})$, where $m \in \NN$. In each case $D \cap N$ has rank $3$ or $4$, so $F^*(G) \cong N C_{2^n}$ for some $n$. Also $N/Z(N)$ has cyclic odd order outer automorphism group and trivial Schur multiplier. It follows from~\cite[Lemma 3.4]{ea01} that $\Aut(N/Z(N)$ also has trivial Schur multiplier in each of these cases. Hence $G \cong C^{2^n} \times H$ for some $H$ with $N \leq H \leq \Aut(N)$.

Suppose that $N \cong SL_2(8)$ and $b$ is the principal block, which has elementary abelian defect group of order $8$ and inertial quotient $C_7$. Note that $\Out(SL_2(8)) \cong C_3$. Then by the above $G \cong H \times C_{2^n}$, where $SL_2(8) \leq H \leq \Aut(SL_2(8))$.  Hence $B$ is Morita equivalent to the principal block of $SL_2(8) \times C_{2^n}$ or $\Aut(SL_2(8)) \times C_{2^n}$.

Suppose that $N \cong SL_2(16)$ and $b$ is the principal block, which has elementary abelian defect group of order $8$ and inertial quotient $C_{15}$. Note that $\Out(SL_2(16)) = 1$. Hence $G=N$ and we are done in this case.

Suppose that $N \cong J_1$ and $b$ is the principal block, which has elementary abelian defect group of order $8$ and inertial quotient $C_7 \rtimes C_3$. Note that $\Out(J_1)=1$, so $G \cong J_1 \times C_{2^n}$ and we are done in this case.

Suppose that $N \cong \ {}^2G_2(3^{2m+1})$ for some $m \in \NN$ and $b$ is the principal block, which has elementary abelian defect group of order $8$ and inertial quotient $C_7 \rtimes C_3$. Then $G \cong H \times C_{2^n}$, where ${}^2G_2(3^{2m+1}) \leq H \leq \Aut({}^2G_2(3^{2m+1}))$. It follows from~\cite[Proposition 3.1]{ea16} that $B$ is Morita equivalent to $b$, and these blocks have the same inertial quotient. In turn $b$ is Morita equivalent to $B_0(\cO (\Aut(SL_2(8)) \times C_{2^n}))$ by~\cite[Example 3.3]{ok97}.

Suppose that $N \cong Co_3$ and $b$ is the non-principal block with elementary abelian defect group of order $8$. Note that $\Out(Co_3)=1$, so $G \cong Co_3 \times C_{2^n}$. By~\cite{kmn11} $b$ is Morita equivalent to $B_0(\cO \Aut(SL_2(8)))$, so $B$ is Morita equivalent to $B_0(\cO (\Aut(SL_2(8)) \times C_{2^n}))$.

We now move on to case (iii) of \thref{qsclassification}. Then $b$ is Morita equivalent to the principal block of $A_5 \times P$ or $A_4 \times P$ for some abelian $2$-group $P$ of rank at most $2$. Since $G/N$ is solvable, \thref{covers_inertial}(\ref{C3_1},\ref{A5factor}) apply, and $B$ is Morita equivalent to a block on our list, a contradiction.

It remains to consider the case that $b$ is nilpotent covered. By~\cite[4.3]{pu11} $b$ is inertial. Since $F^*(G) = NZ(G)O_2(G)$ and $O_2(G)$ has rank at most $2$, it follows that $\Out(O_2(G))$ and $\Out(N)$, and so $\Out(F^*(G))$ are solvable. It follows from \thref{covers_inertial} that $B$ is Morita equivalent to a block in our list, a contradiction. We have covered each possibility for $b$ as given in \thref{qsclassification}, so we have established the Morita equivalences in all cases.

Finally we observe that parts (a) and (b)(iv) are proved in~\cite{McK} and~\cite{am20}.
\hfill $\Box$

\section{Derived equivalences and Brou\'{e}'s abelian defect group conjecture}
\label{broue_section}

In this section we prove \thref{broue}.

We first recall K\"ulshammer-Puig classes of blocks. Let $D$ be a defect group for a block $B$ with Frobenius category $\cF$. Following the presentation in~\cite[Section 8.14]{lin2}, a K\"ulshammer-Puig class is an element of $H^2(\Aut_\mathcal{F}(Q)) \cong H^2(E,k^\times)$. As we have used in \thref{normaldefect}, by~\cite[Theorem 6.14.1]{lin2} the Morita equivalence class of a block with normal defect group is determined by the inertial quotient and the K\"ulshammer-Puig class. If $E$ has cyclic Sylow subgroups for all primes, then $H^2(E,k^\times)$ is trivial, so in this case the Morita equivalence class is determined just by $E$. This is the case for all blocks considered in this paper except for those with inertial quotient $C_3 \times C_3$. Hence, since inertial quotients are preserved by the Morita equivalences in \thref{main_theorem}, in order to prove \thref{broue} for inertial quotients other than $C_3 \times C_3$ it suffices to show that in each of (b)(i)-(iv) all blocks are derived equivalent when the defect groups are isomorphic (the result is trivial for blocks in \thref{main_theorem}(a)). This follows from~\cite[Theorem 4.36]{cr13} since for the groups in each of (b)(i)-(iv), the principal block of the normalizer of a Sylow $2$-subgroup $P$ is Morita equivalent to $P \rtimes E$, where $E$ is the inertial quotient. This proves \thref{broue} when $E \not\cong C_3 \times C_3$.

Now suppose that $B$ has inertial quotient $E \cong C_3 \times C_3$. We must determine whether the Brauer correspondent $b$ of $B$ in $N_G(D)$ is Morita equivalent to $\cO (D \rtimes E)$ or to $\cO (D \rtimes 3_+^{1+2})$. A nonprincipal block of $D \rtimes 3_+^{1+2}$ has just one simple module whilst all other blocks in \thref{main_theorem}(b)(v) have nine simple modules. However by~\cite{ru22} and~\cite[Proposition 5.5]{kr89} $l(B)=l(b)$, so we may distinguish the Morita equivalence class of $b$ by $l(B)$. By~\cite[Theorem 4.36]{cr13} the principal blocks occurring in (b)(v) are derived equivalent, so \thref{broue} follows.

\begin{center} {\bf Acknowledgments}
\end{center}
\smallskip
We thank Benjamin Sambale for some useful observations, and the referee for their careful reading of the manuscript. 

The first author is a member of the editorial board of the Journal of the London Mathematical Society.

\end{document}